\documentclass[11pt, reqno]{amsart}
  \usepackage[all]{xy}
\usepackage{amsmath,amsthm,amsfonts,amssymb,amscd}
\usepackage{graphicx}
\usepackage[latin1]{inputenc}
\usepackage[T1]{fontenc}
\usepackage[all]{xy}
\usepackage{epsfig}
\input{epsf.tex}
\newtheorem{theorem}{Theorem}[section]
\newtheorem{proposition}[theorem]{Proposition}
\newtheorem{lemma}[theorem]{Lemma}
\newtheorem{corollary}[theorem]{Corollary}
\newtheorem{define}[theorem]{Definition}
\newtheorem{remark}[theorem]{Remark}

\def\Empty{}

\catcode`\@=11
\def\section{\@startsection {section}{1}{\z@}{-3.5ex plus -1ex minus
-.2ex}{2.3ex plus .2ex}{\large\bf}}

\def\fnum@figure{{\small Figure \thefigure}}
\def\fakefigure{\def\@captype{figure}}

\long\def\@makecaption#1#2{
    \vskip 10pt
    \def\FCap{#2} \def\NoCap{\ignorespaces}
    \ifx \FCap\NoCap
       \setbox\@tempboxa\hbox{#1}  % This is to avoid the damn colon.
      \else
       \setbox\@tempboxa\hbox{#1: \small \it #2}
    \fi
    \ifdim \wd\@tempboxa >\hsize   % IF longer than one line:
        \unhbox\@tempboxa\par      %   THEN set as ordinary paragraph.
      \else                        %   ELSE  center.
        \hbox to\hsize{\hfil\box\@tempboxa\hfil}
    \fi}

\def\@oddhead{\hbox{}\rightmark \hfil \rm\thepage}% Right heading.
\def\sectionmark#1{\markright {\sc{\ifnum \c@secnumdepth >\z@
      \S\thesection.\hskip 1em\relax \fi #1}}}

\catcode`\@=12

\def\oplabel#1{
  \def\OpArg{#1} \ifx \OpArg\Empty {} \else
  	\label{#1}
  \fi}

\newlength{\saveu}

\catcode`\@=11

\newcommand{\eps}{\epsilon}
\newcommand{\RR}{\mathbb{R}}

\newcommand{\oo}{\mbox{$\mathcal O$}}

\newcommand{\mi}{\mbox{$\widetilde M$}}

\newcommand{\wwp}{\mbox{$\widetilde \Phi$}}

\newcommand{\fs}{\mbox{${\mathcal F} ^s$}}
\newcommand{\fu}{\mbox{${\mathcal F} ^u$}}
\newcommand{\fss}{\mbox{${\mathcal F} ^{ss}$}}
\newcommand{\fuu}{\mbox{${\mathcal F} ^{uu}$}}

\newcommand{\fns}{\mbox{${\widetilde {\mathcal F}}^s $}}

\newcommand{\fnu}{\mbox{${\widetilde {\mathcal F}}^u $}}

\catcode`\@=12

\begin{document}

\title[\centerline{Orbital equivalence classes of finite coverings of geodesic flows}]{Orbital equivalence classes of finite coverings of geodesic flows}

%\footnote{Preliminary version. Comments will be appreciated.}}

\author{Thierry Barbot}

\author{S\'{e}rgio R. Fenley}
\thanks{Sergio Fenley's reseach was partially supported by Simons foundation
grants 280422 and 637554, and by National Science Foundation
grant DMS-2054909.}

\address{Thierry Barbot\\ Avignon Universit\'e\\
LMA, Campus Jean-Henri Fabre\\
301, rue Baruch de Spinoza\\
F-84 916 AVIGNON Cedex 9.}

\email{thierry.barbot@univ-avignon.fr}

\address{S\'ergio Fenley\\Florida State University\\
Tallahassee\\FL 32306-4510, USA}

\email{fenley@math.fsu.edu}

\maketitle

%\centerline{PRELIMINARY VERSION}
\vskip .2in

{\small
\noindent
{{\underline {Abstract}} $-$
Let $M$ be a closed $3$-manifold admitting a finite cover of degree $n$ along the fibers over the unit tangent bundle of a connected closed surface. 
We prove that if $n$ is odd, there is only one Anosov flow on $M$ up to orbital equivalence, and if $n$ is even, there are exactly two orbital equivalence classes of Anosov flows on $M$.
}
}

\vspace{1cm}

\tableofcontents

\section{Introduction}
Let $\Sigma$ be a closed connected orientable surface of genus $g>1$, equipped with a negatively curved Riemannian metric, and let $T^1\Sigma$ be the unit tangent bundle. 
The geodesic flow $\Phi_0^t$ on $T^1\Sigma$ is a typical {\em Anosov flow} (see Section \ref{sub:geodesic} for a definition of the geodesic flow, and Definition \ref{def:anosov}
for a definition of Anosov flows).

E. Ghys, in \cite{Gh}, proved that, up to finite covers, geodesic flows are essentially the only possible Anosov flows on circle bundles over a closed surface. More precisely, if $\Phi^t$ is
an Anosov flow on a closed manifold $M$ admitting a circle fibration over the closed surface $\Sigma$, then there is a finite covering $q: M \to T^1\Sigma$ 
such that $\Phi^t$ is a reparametrization of the lifting in $M$ by $q$ of the geodesic flow on $T^1\Sigma$ (see Theorem \ref{ghys}).

Observe that the finite covering $q$ is {\em along the fibers} (see Section \ref{sub:along}) $-$ roughly this means that both fibrations
have the same base space $\Sigma$. This implies that if $M$ is homeomorphic the unit tangent bundle $T^1\Sigma$ itself, then
$q$ is a homeomorphism, and that, in this case, $\Phi^t$ is merely orbitally equivalent to the geodesic flow itself (see Definition \ref{def:orbital}).

This is not true in the general case. More precisely: let $M$ be a closed $3$-manifold admitting a circle fibration over $\Sigma$. Ghys' Theorem implies that in order
to support some Anosov flow, $M$ must be a finite cover of $T^1\Sigma$. The degree $n$ of this finite cover is a topological invariant of $M$: any other finite covering $q: M \to T^1\Sigma$ has 
degree $n$ (see beginning of Section \ref{sub:along}). For a fixed interger $n$, there are infinitely many such finite coverings, providing different Anosov flows on $M$. Hence a natural question is to ask 
if different finite coverings can provide different orbital equivalence classes of Anosov flows? The answer is given by Theorem \ref{th:maincircle}, that we restate here:

\bigskip

\noindent{\bf Main Theorem :} 
  {\em Let $q_0: M \to T^1(\Sigma)$ be a finite covering of degree $n$ along the fibers. Then, if $n$ is odd, there is only one Anosov flow on $M$ up to orbital equivalence.
  If $n$ is even, there are exactly two orbital equivalence classes of Anosov flows.}

\medskip

The proof of the Main Theorem has a dynamical part which is
mainly contained in Ghys' Theorem. After that the proof
is reduced in this article to 
a study of the action of the (extended) mapping class group Mod$^\pm(\Sigma)$ on
the set $\mathfrak{G}_n$ of subgroups of $\widetilde{\Gamma} = \pi_1(T^1\Sigma)$ corresponding to finite coverings of degree $n$ along the fibers. 
More precisely, there is an exact sequence 
\begin{equation*}
  0 \to \mathbb{Z} \to \widetilde{\Gamma} \to \pi_1(\Sigma) \to 1
\end{equation*}
and elements of $\mathfrak{G}_n$ are subgroups of $\widetilde{\Gamma}$ of index $n$ such that the projection on $\pi_1(\Sigma)$ is surjective (see Definition \ref{def:GGn}). 

It is known (and recalled in Section \ref{sub:G_n}) that the mapping class group Mod$^\pm(\Sigma)$ acts naturally on $\mathfrak{G}_n$. In Proposition \ref{pro:Fgroup} 
we prove that there is a one-to-one correspondence between orbital equivalence classes of Anosov flows in $M$ and Mod$^\pm(\Sigma)$-orbits in $\mathfrak{G}_n$.

Hence, the core of the proof is Proposition \ref{pro:computeMod}, where we compute explicitly how the Lickorish generators of Mod$(\Sigma)$ act on $\mathfrak{G}_n$, and which allows us (with a little more work)
to compute the number of Mod$^\pm(\Sigma)$-orbits in $\mathfrak{G}_n$ according to the parity of $n$.

Our Main Theorem is closely related to Giroux's classification of universally tight
contact structures, with wrapping number $-n$, on a circle bundle $M$ (\cite{Giroux}). 
See Remark \ref{rk:giroux} for more details.

In addition, in Section \ref{sub:along} we describe the space of isotopy classes of Anosov flows on a given circle bundle over a closed surface (Corollary \ref{cor:ModTransitive}). There are infinitely many, but there is only one up to vertical Dehn twists as defined in Remark \ref{rk:homeofibre}.

Finally we note the recent work of Barthelme, Mann and collaborators
on classifying transitive Anosov (and pseudo-Anosov) flows up to
orbital equivalence using the knowledge of the spectrum
\cite{Ba-Ma,BFeM,BFrM}. The {\em spectrum} is the set of 
free homotopy classes of curves in $M$ which are represented
by periodic orbits of the flow.
In particular perhaps it might  be possible to prove the orbital
equivalence result of Subsection \ref{sub:along} (Proposition
\ref{pro:Fgroup}) using the results of Barthelme, Mann
and collaborators.

\section{Background}

\subsection{Anosov flows $-$ definitions}

\begin{define}{(Anosov flow)}\label{def:anosov}
Let $\Phi^t$ be a non-singular $C^k$-flow ($k \geq 1$) on a closed connected 3-manifold $M$, equipped with an auxiliary Riemannian metric. We say
that $\Phi^t$ is  Anosov if the tangent bundle of $M$ admits a continuous $d\Phi^t$ - invariant splitting:
$$TM = \mathbb{R}\Phi \oplus E^{ss} \oplus E^{uu}$$
such that:

%\begin{itemize}

%\item
- $\mathbb{R}\Phi$ is the one-dimensional bundle tangent to the orbits of the flow,

- There are two positive real numbers $a, C$, such that, for every vector $v^s$ in $E^{ss}$ (respectively $v^u$ in $E^{uu}$) and for every $t >0$, the following
inequalities hold:
\begin{eqnarray*}
  \| d\Phi^t(v^s) \|  &\leq& Ce^{-at}\| v^s\| \\
  \| d\Phi^t(v^u) \|  &\geq& \frac{1}{C} e^{at}\| v^u\|
\end{eqnarray*}

\end{define}

The invariant bundles $E^{ss}$, $E^{uu}$ are called the \emph{strong stable and strong unstable bundles}, respectively. They are
usually only H\" older continuous, nevertheless they are 
uniquely integrable, tangent to one-dimensional foliations $\fss$ and $\fuu$, called
the \emph{strong stable and unstable foliations} \cite{An,Ha-Ka}.

The \emph{weak stable and unstable bundles} $E^s := \mathbb{R}\Phi \oplus E^{ss}$ and $E^u := \mathbb{R}\Phi \oplus E^{uu}$ are uniquely integrable too,
the foliations $\fs$ and $\fu$ they integrate to are the \emph{weak stable and unstable foliations.} These foliations are $C^1$.
In fact these foliations are even $C ^{1+\text{Zygmund}}$ in the
case that the flow preserves some volume form on $M$
and the flow is $C^k$ with $k \geq 3$  (Theorem 3.1 of
\cite{hurderkatok}).

Sometimes we will assume that $M$ is orientable and/or
the foliations $\fs$, $\fu$ are transversely orientable.
This can always be achieved in a cover of order at most $4$.

The intersection
foliation $\Phi = \fs \cap \fu$ is an oriented $1$-dimensional foliation, whose leaves are the orbits of $\Phi^t$, oriented by the time direction.
The weak foliations $\fs$, $\fu$ only depend on $\Phi$ and not on the parametrization of the flow, contrary to the strong foliations $\fss$, $\fuu$.

In this article we will use the following terminology, which
may be different from other articles:

\begin{define}
An \emph{Anosov foliation} is a foliation $\Phi = \fs \cap \fu$, oriented by the time direction, where  $\fs$, $\fu$ are the weak stable and unstable
foliations of some Anosov flow.
\end{define}

We also remind the following definition, for the reader's convenience:

\begin{define}\label{def:orbital}
 A $C^r$-\emph{orbital equivalence} between two flows $(M, \Phi^t)$ and $(N, \Psi^s)$ is a $C^r$-diffeomorphism $f: M \to N$ mapping oriented orbits onto oriented
 orbits: there exists a continuous map $u: \mathbb{R} \times M \to \mathbb{R}$, increasing with the first factor, and such that, for every $(t,x)$ in $\mathbb{R} \times M$ we have:
 $$f(\Phi^t(x)) = \Psi^{u(t,x)}(f(x)).$$
 The flows are then $C^r$-\emph{orbitally equivalent}. If moreover the map $u$ satisfies $u(t,x)=t$, then the orbital equivalence is a \emph{($C^r$-)conjugation}.
Here $r$ could be $0$, in which case $f$ is only a homeomorphism.
\end{define}

Observe that the orbital equivalence maps the foliation $\Phi$ onto the foliation $\Psi$. Orbital equivalences between Anosov flows can therefore
be defined as conjugacies between Anosov foliations preserving the orientation.
In this article we will consider orbital equivalences between
Anosov flows and therefore omit the reference to Anosov flows.
Orbital equivalences map weak leaves onto weak leaves.

From now on,
$\Phi^t$ denotes a $C^k$-Anosov flow ($k \geq 1)$ on a closed $3$-manifold $M$. % with transversely orientable weak foliations.

\subsection {Orbit space and leaf spaces of Anosov flows}\label{sub:orbit}

\vskip .05in
\noindent
%\underline {Notation/definition:} \
We denote by $\pi: \mi \to M$ the universal covering of $M$, and by $\pi_1(M)$ the fundamental group of $M$,
considered as the group of deck transformations on $\mi$.
The foliations lifted to $\mi$ are
denoted by $\fns, \fnu$.
If $x \in M$ let $\fs(x)$ denote the leaf of $\fs$ containing
$x$.  Similarly one defines $\fu(x)$
and in the
universal cover $\fns(x), \fnu(x)$.
If $\theta$ is a leaf of $\Phi$, we denote by
$\fs(\theta)$, $\fu(\theta)$ the weak leaves containing $\theta$.
Let also $\wwp^t$ be the lifted flow to $\mi$. We adopt similar conventions denoting by $\fns(\tilde{\theta})$, $\fnu(\tilde{\theta})$ the weak
stable and unstable leaves
of $\tilde{\theta}$.

We denote by $\mathcal O$ the orbit space $\mi/\widetilde{\Phi}$ and by $\Theta: \mi \to \oo$ the quotient map. It is remarkable
that $\mi$ is always homeomorphic to $\mathbb{R}^3$, $\oo$ diffeomorphic to $\mathbb{R}^2$, and that $\Theta$ is a $\pi_1(M)$-equivariant $\mathbb{R}$-principal fibration (\cite{Ba1}, \cite{Fe-Mo} ). We denote by $\oo^s$, $\oo^u$ the $1$-dimensional foliations induced by $\fns$, $\fnu$ on
$\oo$. The leaves of these foliations are properly
embedded lines.

We denote the leaf spaces $\mi/\fns$, $\mi/\fnu$ by respectively $Q^s$, $Q^u$. Observe that they are canonically identified with
the leaf spaces $\oo/\oo^s$, $\oo/\oo^u$.
A particular case of interest for us is the case of \emph{$\mathbb{R}$-covered Anosov flows,} i.e. the case where $Q^s$ or $Q^u$ is homeomorphic
to the real line. In fact if one of them is homeomorphic to the real line,
then both of them are \cite{Ba1}.
We fix homeomorphisms of these with $\mathbb{R}$.
%Since $\fs$ and $\fu$ are transversally oriented, the
%leaf spaces are naturally equipped with
This induces total order in $Q^s, Q^u$ (both homeomorphic to $\mathbb{R}$).
If $\fs, \fu$ are transversely oriented this total order is preserved
by deck transformations.
% a total order, which is identified with the usual total order $<$ on $\mathbb{R}$.

In addition the following is true:
% here we do not assume that
%$\fs, \fu$ are transversely orientable.

\begin{theorem}(\cite[Th\'eor\`eme $2.10$]{Ba2})\label{thm:QQ}
  Let $(M, \Phi^t)$ be an $\mathbb{R}$-covered Anosov flow. Assume that $(M, \Phi^t)$ is not orbitally equivalent to the suspension of an Anosov diffeomorphism of the torus. Then, the map $\Upsilon: \oo \to Q^s \times Q^u$ sending an orbit $\tilde{\theta}$ to the pair
  $(\fns(\tilde{\theta}), \fnu(\tilde{\theta}))$ is an homeomorphism onto its image, which is:
  $$\Omega := \{ (x,y) \in Q^s \times Q^u \; | \; \alpha(x) < y < \beta(x) \}$$
 where $\alpha$ and $\beta$ are
  two  homeomorphisms from $Q^s$ into $Q^u$ satisfying, for every $\gamma$ in $\pi_1(M)$:
$$    \alpha \circ \gamma = \gamma \circ \alpha \mbox{ and } \beta \circ \gamma = \gamma \circ \beta \mbox{ if }
    \gamma \mbox{ preserves the orientation of } Q^s \mbox{ and } Q^u $$
$$    \alpha \circ \gamma = \gamma \circ \beta \mbox{ and } \beta \circ \gamma = \gamma \circ \alpha \mbox{ if }
    \gamma \mbox{ reverses the orientation of } Q^s \mbox{ and } Q^u $$
\end{theorem}

\begin{remark}
\label{rk.TT}

{\em The $\RR$-covered Anosov flows in Theorem \ref{thm:QQ} are
called \emph{skewed} Anosov flow. We refer the reader to
\cite{Ba1,Fe1} for a more detailed study of these flows and their
properties. For example, it follows from these papers that for every $\gamma$ in $\pi_1(M)$, $\gamma$ preserves the orientation of $Q^s$ if anf only if it preserves the
orientation of $Q^u$, so that the two cases considered at the end of Theorem \ref{thm:QQ} indeed cover all the possible cases.

Here we first stress that $\alpha, \beta$ are homeomorphims, and not
just continuous maps, which is a big part of the content of the theorem.
In particular it follows from Theorem \ref{thm:QQ} that
%, when $\fs$ and $\fu$ are transversely oriented,
the composition $\tau_s = \alpha^{-1} \circ \beta$ is a homeomorphism without fixed points of $Q^s$ into itself satisfying
for every $\gamma$ in $\pi_1(M)$:
$$    \tau_s \circ \gamma = \gamma \circ \tau_s \mbox{ if }
    \gamma \mbox{ preserves the orientation of } Q^s \mbox{ and } Q^u $$
$$    \tau_s \circ \gamma = \gamma \circ \tau_s^{-1}  \mbox{ if }
    \gamma \mbox{ reverses the orientation of } Q^s \mbox{ and } Q^u $$

We can furthermore define a map  $\zeta: \oo \to \oo$ which is conjugated by $\Upsilon$ to the map:
$$(x,y) \mapsto (\alpha^{-1}(y), \beta(x)).$$

This map $\zeta$ satisfies:
$$\zeta \circ \gamma = \gamma \circ \zeta \mbox{ if }
    \gamma \mbox{ preserves the orientation of } Q^s \mbox{ and } Q^u $$

$$\zeta \circ \gamma = \gamma \circ \zeta^{-1}  \mbox{ if }
    \gamma \mbox{ reverses the orientation of } Q^s \mbox{ and } Q^u $$

Observe that the map $\zeta$ permutes the foliations $\oo^s$ and $\oo^u$.
By an abuse of notation we will identity $\oo$ and $\Omega$ by the
homeomorphism $\Upsilon$. This induces an action of $\zeta$ on $\Omega$.
}
\end{remark}

\begin{remark}
\label{rk.holder}
{\em   The homeomorphisms $\zeta$, $\tau_s$, $\alpha$ and $\beta$ are H\"older continuous (see \cite[Proposition $2.1$]{Ba3}).}
\end{remark}

\begin{remark}
{\em 
The image under $\Upsilon$ of a stable leaf $\oo^s(\tilde{\theta})$ is a vertical segment contained in $\Omega$ that we will call later a {\em stable leaf} in $\Omega$. 
Similarly, images under $\Upsilon$ of unstable leaves $\oo^u(\tilde{\theta})$ will be horizontal segments that we will call later {\em unstable leaves} in $\Omega$.}
\end{remark}

\begin{corollary}\label{cor:powerT}
 Let $(M, \Phi^t)$ be a skewed $\mathbb{R}$-covered Anosov flow.
Then, the only homeomorphisms from $\oo$ into itself commuting with the action of $\pi_1(M)$ are the powers $\zeta^{k}$ ($k \in \mathbb{Z})$ of $\zeta$.
If the homeomorphism is induced by a self orbit equivalence of $\Phi$, then
it is of the form $\zeta^{2k}$, where $k \in \mathbb{Z}$.
\end{corollary}

\begin{proof}
%According to Remark \ref{rk.TT}, every even power $T^{2k}$ commutes
% with $\pi_1(M)$.
Let us consider some homeomorphism $f: \oo \to \oo$ commuting with the action of $\pi_1(M)$.
We will abuse notation and also denote by $f$ the action induced
on $\Omega$
by its conjugate under $\Upsilon$.

Let $\gamma$ be a non-trivial element of $\pi_1(M)$ fixing a point $(x,y)$ of $\Omega$. Then, the other fixed points of $\gamma$ are the elements of the orbits
of $(x,y)$ under $\zeta$. Therefore, there is some integer $k$ such that:
$$f(x,y) = \zeta^{k}(x,y)$$
Hence, $(x,y)$ is a fixed point of $\zeta^{-k} \circ f$.
%It is already a fixed point of $\gamma$, more precisely, a hyperbolic fixed point. Without loss of generality, by replacing $\gamma$ by $\gamma^{-1}$, one can assume that the stable leaf of $(x,y)$ for $\gamma$ is the horizontal $y= Cte$,
%and the unstable leaf the vertical $\{ x=Cte \}$.
%
%If $k$ is even, $\gamma$ commutes with  $T^{-k} \circ f$, and if $k$ is odd, $\gamma$ and  $T^{-k} \circ f$ anticommutes. It follows that in
%the first case,  $T^{-k} \circ f$ preserves individually the stable and unstable leaves $\oo^s(x,y)$ and $\oo^u(x,y)$, whereas in the second case, these leaves are
%permuted. In the first case, we say that $f$ is even at $(x,y)$, and in the second case that $f$ is odd at $(x,y)$.

 Let $(x_1, y_1)$ fixed by some $\gamma_1$ non trivial
and $(x_1,y_1)$ sufficiently close to
$(x,y)$, so that $(x_1, y)$ and $(x, y_1)$ are in $\Omega$, %the stable leaf of $\gamma_1$ at $(x_1, y_1)$ intersects the unstable leaf of $\gamma$ at $(x,y)$ and vice versa, 
and in addition $\zeta^{-k} \circ f(x_1,y_1)$ is very close to $(x_1,y_1)$ also.
The previous paragraph implies that there is $j$ so that
$\zeta^{-j} \circ f(x_1,y_1) = (x_1,y_1)$. Hence the stable leaf  of $\zeta^{-k}(x_1,y_1)$
intersects the unstable leaf of $\zeta^{-j}(x_1,y_1)$ and vice versa. This
immediately implies that $j = k$. In particular the arbitrarily near
is uniform, that is there is $\eps > 0$ so that if the
orbit of $(x,y)$ and the orbit of $(x_1,y_1)$ are $\eps$ close
at some point then the above happens.

%, then it is clear
%that $f$ is even at $(x_1,y_1)$ if and only if it is even at $(x,y)$. Indeed, if $f$ is even at $(x,y)$, it maps $\oo^s(x,y)$ on
%$T^k\oo^s(x,y)$, and maps
%$\oo^u(x_1,y_1)$ to some line intersecting$T^k\oo^s(x,y)$ at one and only one point. Therefore, it cannot be a stable leaf: $f$ must map
%%$\oo^u(x_1,y_1)$ onto some unstable leaf, and therefore $f$ must be even at $(x_1,y_1)$. The case where $f$ is odd at $(x,y)$ is similar.
%Moreover, it follows that the integer $k_1$ for which we have $f(x_1,y_1) = T^{k_1}(x_1,y_1)$ is equal to $k$. Once again, we treat only the even
%case, the odd case being similar. And the claim follows simply from the fact that since $\oo^u(x_1,y_1)$ intersects $\oo^s(x,y)$, the image $f(\oo^u(x_1,y_1)) = T^{k_1}(x_1, y_1)$ must intersects $f(\oo^s(x,y)=T^k(x,y)$, and the equality $k=k_1$ follows.

Since $\Phi$ is $\mathbb R$-covered, it is topologically transitive
\cite{Ba1}, and therefore the union of periodic orbits are dense in $M$. It means that
fixed points of non trivial elements of  $\pi_1(M)$ are dense in $\Omega$. It follows that elements of $\Omega$ with non-trivial  $\pi_1(M)$-stabilizers
can be connected by a finite sequence of elements with non-trivial  $\pi_1(M)$-stabilizers such that two successive elements of the sequence are sufficiently close one to the other in the preceding meaning.

As a conclusion, the integer $k$ at elements of $\Omega$ with non-trivial  $\pi_1(M)$-stabilizer is the same at each of them.
In other words, there is an integer $k$ such that $\zeta^{-k} \circ f$ preserves every element of $\Omega$ with non-trivial  $\pi_1(M)$-stabilizer. Since the union of these elements is dense in $\Omega$, the equality $f = \zeta^k$ holds everywhere.
This proves the first statement of the corollary.

To prove the second statement notice that if the
homeomorphism $f$ comes from
an orbit equivalence, then it sends the weak stable foliation
to itself and likewise for the weak unstable foliation.
By Remark \ref{rk.TT}, it now follows that
$f$ must be an even power of $\zeta$.
The corollary follows.
\end{proof}

%We will see (Theorem \ref{thm:ba1}) that, under the hypothesis of Theorem \ref{thm:QQ}, there is always an orbital equivalence between $\Phi$ ans its inverse realizing $\zeta$.\footnote{T: I have removed the discussion you did concerning the geodesic flow, and postpone to the next new section.}

We remark that more generally one can consider orientation reversing orbit equivalences as well: they are ``weak'' orbit equivalences which reverse the direction of the flow. 
%For example suppose that $\Phi$ is the geodesic flow of a closed, orientable hyperbolic surface $S$, hence $\Phi$ is a flow in $M = T^1 S$. Let $h$ be the map in $M$ which sends a vector $v$ based at $p$ in $S$ to the vector$-v$ in $T_p S$. In other words $h$ rotates the vectors by $\pi$ keeping their base points. It is easy to see that $h$ sends orbits of the geodesic flow (oriented geodesics in $S$) to orbits of $\Phi$. That is, it sends an oriented geodesic to the same geodesic traversed in the opposite direction. The map $h$ induces a homeomorphism $f$ from the orbit space to itself. One can show that one can choose $f$ so that the induced map $f$ in $\oo$ is $T$. It also commutes with every deck transformation: here we use that $S$ is orientable, hence the weak stable and unstable foliations are transversely orientable. This implies that every $\gamma$ in $\pi_1(M)$ commutes with $f = T$ by Remark \ref{rk.TT}. So this is an example of a map $f$ commuting with every deck transformation, but $f$ is an odd power of $T$.Notice that $h$ sends the weak stable foliation to the weak unstable foliation and vice versa.

%Typical examples of contact Anosov flows are geodesic flows of negatively curved Riemannian surfaces.

\subsection{Recovering the Anosov flow from the orbit space}\label{sec:orbitspace}
In this Section, we see how, from the data of the orbit space $\oo$ and the $\pi_1(M)$-action on it, to recover the Anosov flow itself up
to orbital equivalence.
The flow $(M, \Phi^t)$ is not assumed to be $\mathbb{R}$-covered, nor
topologically transitive.

\begin{theorem}(\cite[Th\'eor\`eme $3.4$]{Ba1})\label{thm:ba1}
  Let $(M, \Phi^t)$ and $(N, \Psi^s)$ be two $C^k$-Anosov flows on closed connected $3$-manifolds ($k \geq 1)$. Assume that there exists an isomorphism
  $\rho: \pi_1(M) \to \pi_1(N)$ and an equivariant $C^r$-diffeomorphism  $\bar{f}: \oo \to Q^\Psi$ between the orbit spaces ($1 \leq r \leq k$).
  Then, $\bar{f}$ lifts to an equivariant $C^r$-diffeomorphism between
  the universal coverings. In other words, either $(M, \Phi^t)$ is orbitally equivalent to $(N, \Psi^s),$ or $(M, \Phi^t)$ is orbitally equivalent
  to $(N, \Psi^{-s})$ (i.e. the second flow with the time direction reversed). If moreover $\bar{f}$ maps the stable foliation $\oo^s(\Phi)$ onto the stable
  foliation $\oo^s(\Psi)$, then the first case occurs, i.e. $(M, \Phi^t)$ is orbitally equivalent to $(N, \Psi^s).$
\end{theorem}

It follows in particular that, according to Remark \ref{rk.TT}, an $\mathbb{R}$-covered Anosov flow
admitting no global cross-section is weak orbitally equivalent
to its own inverse: the map $\zeta$ defined in  Remark \ref{rk.TT} lifts to some weak orbital equivalence $F_0: M \to M$ between $\Phi$ and its inverse. Moreover, $F_0$ is  isotopic to the identity.

Observe also the following consequence of Corollary \ref{cor:powerT}:

\begin{corollary}\label{cor:powerT2}
 Let $(M, \Phi^t)$ be a skewed $\mathbb{R}$-covered Anosov flow.  Let
$F: M \to M$ be an orbital equivalence between $\Phi$ and itself. Assume that $F$ is  isotopic to the identity. Then, $F$ is isotopic along $\Phi$ to some power $F_0^{2k}$ of $F_0$.
\end{corollary}

This follows because the orbital equivalence has a lift to $\mi$ which
commutes with all deck translations.

\section{Classification of finite coverings of geodesic flows up to isotopy and orbital equivalence}

\subsection{Geodesic flows}\label{sub:geodesic}
In this Section, $\Sigma$ is a closed connected oriented surface of genus $g>1,$ we denote by $\bar{\Gamma}$ its fundamental group,
and by $p^0: M_1(\Sigma) \to \Sigma$ the \emph{positive} projective tangent bundle of $\Sigma,$ i.e. the quotient of the tangent bundle
with the zero section removed by the relation identifying two vectors if they are proportional up to a \emph{positive} real number.
This definition avoids the choice of a peculiar Riemannian metric, but $M_1(\Sigma)$ clearly identifies with the unit
tangent bundle for any Riemannian metric on $\Sigma$.

According to \cite{An}, the geodesic flow of any negatively curved metric on $\Sigma$ is an Anosov flow on $M_1(\Sigma).$
Moreover, since the Teichm\"uller space Teich$(\Sigma)$
is connected, and since the space of negatively curved metrics in a given conformal class is connected, any pair of negatively curved metrics on $\Sigma$
can be joined by a path of negatively curved metrics. It follows from structural stability of Anosov flows that Anosov geodesic flows
on $M_1(\Sigma)$ are isotopic one to the other. Hence, we can speak of \emph{the Anosov geodesic flow} of $\Sigma$.
%In other words, the geodesic flow admits a huge group of symmetry up
%to isotopy: if we denote it by $\Phi^_0$, for any homeomorphism $f$ of $M_1(\Sigma)$, the pull-back flow $f^\ast(\Phi^t_0)$ is isotopic to $\Phi^t_0$.

Let us review alternative ways to define geodesic flows for hyperbolic metrics, each of them being useful in the rest of the paper:

\subsubsection{Geodesic flows as algebraic flows.}\label{sub.algebraic}  Let $\widetilde{\operatorname{PSL}}(2, \mathbb{R})$ be
the universal covering of PSL$(2, \mathbb{R})$, and let $\widetilde{\Gamma}$ be the full preimage in $\widetilde{\operatorname{PSL}}(2, \mathbb{R})$  of the uniform
lattice $\bar{\Gamma}$. Then, $ M_1(\Sigma)$ is diffeomorphic to the quotient $\widetilde{\Gamma}\backslash\widetilde{\operatorname{PSL}}(2, \mathbb{R})$,
and the geodesic flow is the flow induced by the action on the right of the $1$-dimensional Lie subgroup whose projection in  PSL$(2, \mathbb{R})$
is the  subgroup D represented by
diagonal matrices of the form:

$$
\left(
  \begin{array}{cc}
    e^t & 0 \\
    0 & e^{-t} \\
  \end{array}
\right)
$$

More generally, every finite covering of the geodesic flow has a similar description, where $\widetilde{\Gamma}$ is replaced by some finite index subgroup of itself.

\subsubsection{Geodesic flow on the space of triples.}\label{sub.triple} This construction is described by M. Gromov in \cite{3remarks} and therein attributed to M. Morse.
Realize $\bar{\Gamma}$ as a torsion-free uniform lattice in PSL$(2, \mathbb{R})$, hence as a discrete group
of projective transformations of the (oriented) circle $\mathbb{R}P^1$.
For now and future use let

$$X \ = \ \{ \ (x,y,z) \in (\mathbb{R}P^1)^3,
\  x \not = y \not = z \not = x,
\ x < y < z \ \ {\rm for \ the \ cyclic \ order \ of} \ \mathbb{R}P^1 \ \}$$

%$, $y$, and $z$ are pairwise
%distinct, and such that
%$x < y < z$ for the cyclic order on $\mathbb{R}P^1$.
\noindent
For every $(x,y,z)$ in $X$ let $p$ be the unique point in the hyperbolic plane $\mathbb{H}^2$ lying in the geodesic $(xy)$ of extremities $x$ and $y$ such
that the geodesic ray starting from $p$ and going to $z$ is orthogonal to $(xy)$, and let $v$ be the unit vector tangent to $(xy)$ at $p$ pointing towards $x$.
%\footnote{S: Changed $y$ to $x$, and so on. So that in the next subsection, stable is associated with fixed  $x$ and not fixed $y$.}
The map $(x,y,z) \mapsto (p, v)$ identifies $X$ with the unit tangent bundle of the hyperbolic plane $T^1\mathbb{H}^2$. It follows that the diagonal action
of $\pi_1(M)$ on $X$ is free and properly discontinuous, and that the quotient space is homeomorphic to $M_1(\Sigma)$. Moreover, the geodesic flow on
$M_1(\Sigma)$ corresponds to the flow on $X$ preserving $x$ and $y$ and moving $z$ from $y$ to $x$. Another choice of realization of $\bar{\Gamma}$
as an uniform lattice in PSL$(2, \mathbb{R})$ leads to the same action on $\mathbb{R}P^1$ up to topological conjugacy, hence to the same flow
up to orbital equivalence.

\subsubsection{Geodesic flows on the projectived tangent bundle over the band.}\label{sub.W}  Let us consider once more the universal covering $\widetilde{\operatorname{PSL}}(2, \mathbb{R})$  and the uniform lattice  $\widetilde{\Gamma}$.
Observe that $\widetilde{\Gamma}$ acts naturally on the (oriented) universal covering $\widetilde{\mathbb{R}P}^1$ of $\mathbb{R}P^1$.
The center of $\widetilde{\operatorname{PSL}}(2, \mathbb{R})$ is the group of deck transformations of
the universal covering $\widetilde{\operatorname{PSL}}(2, \mathbb{R}) \to$ PSL$(2, \mathbb{R})$; it is a cyclic group generated by an increasing map
$\delta: \widetilde{\mathbb{R}P}^1 \to \widetilde{\mathbb{R}P}^1$.

The identification between $X$ (of the previous subsection)
and $T^1\mathbb{H}^2$ makes clear that the lifting in $X$ of the weak stable leaves
are the level sets of the projection map $(x,y,z) \mapsto x$. Therefore, the leaf space of this weak stable foliation is $\mathbb{R}P^1$.
Hence, the leaf spaces $Q^s_0$ and $Q^u_0$ for the geodesic flow are both $\widetilde{\Gamma}$-equivariantly isomorphic to $\widetilde{\mathbb{R}P}^1$, and
the image of the map $\upsilon: \oo_0 \to Q^s_0 \times Q^u_0$ appearing in Theorem \ref{thm:QQ} is $\widetilde{\Gamma}$-equivariantly isomorphic to the region 
in $\widetilde{\mathbb{R}P}^1 \times \widetilde{\mathbb{R}P}^1$ between the graphs of the identity map of $\widetilde{\mathbb{R}P}^1$ and the map $\delta.$

\begin{remark}\label{rk:actionQ}
{\em
Let $p_\ast^0: \widetilde{\Gamma} \to \bar{\Gamma}$ be the projection map. Let $\bar{\gamma}$ be a non-trivial element of $\bar{\Gamma}$. It fixes two points
in $\mathbb{R}P^1$, hence it admits a lift $\widetilde{\gamma}$ in $\widetilde{\Gamma}$ that admits fixed points in $\widetilde{\mathbb{R}P}^1$.
More precisely, $\bar{\gamma}$ admits an attracting fixed point, and a repelling fixed point.
Therefore, $\widetilde{\gamma}$ admits a $\delta$-orbit of attracting fixed points, and a $\delta$-orbit of repelling fixed points.
Other elements of $(p^0_\ast)^{-1}(\bar{\gamma})$ are elements of $\widetilde{\Gamma}$ of the form $\widetilde{\gamma}\delta^k$ for some integer $k$. It follows
that $\widetilde{\gamma}$ is the unique element $(p^0_\ast)^{-1}(\bar{\gamma})$ admitting fixed points in $\widetilde{\mathbb{R}P}^1$.

We thus have defined a canonical section $\sigma_0: \bar{\Gamma} \to \widetilde{\Gamma}$, the one associating to every element $\bar{\gamma}$
the unique element in its preimage by $p_\ast^0$ fixing at least one point in $\widetilde{\mathbb{R}P}^1$.
This map is not a group homomorphism. Actually, it defines a cocycle $c: \bar{\Gamma} \times \bar{\Gamma} \to \mathbb{Z}$, where
$c(\bar{\gamma}_1, \bar{\gamma}_2)$ is the unique integer $k$ satisfying:
$$\sigma_0(\bar{\gamma}_1\bar{\gamma}_2) = \delta^k\sigma_0(\bar{\gamma}_1)\sigma_0(\bar{\gamma}_2)$$
This cocycle represents a cohomology class in H$^2(\bar{\Gamma}, \mathbb Z)$ which is the Euler class.
This cocycle  is not trivial, meaning it is not a coboundary. As we will also see later, this  will not be the case for
its projection in $\mathbb{Z}/n\mathbb{Z}$, when $n$ divides $2g-2$.
Actually, %we will show later that
as proved in Lemma \ref{le:cvanishes}, this cocycle takes value in $\{-1, 0, +1 \}$,
meaning that it represents a \emph{bounded} cohomology class, which is the bounded Euler class (see \cite{Gheuler}).
%We will do some computations of this cocycle in the proof of the Claim after Theorem \ref{lick}. These computations indicate why the cocycle only has values in $\{ -1, 0, 1 \}$. In fact the cocycle has value zero if 1) the axes of $\bar{\gamma_1}, \bar{\gamma_2}$ intesect each other; or 2) the axes of $\bar{\gamma_1}, \bar{\gamma_2}$ do not intersect, but the directions agree. If the directions do not agree, then the cocycle is either $1$ or $-1$. See an explicit computation of this after Theorem \ref{lick}. The other computations are left to the reader.
}
\end{remark}

The following lemma will be extremely useful:

\begin{lemma}\label{le:cvanishes}
 Choose an arbitrary hyperbolic metric on $\Sigma$, so that $\bar{\Gamma}$ is realized as a uniform lattice in PSL$(2, \mathbb{R})$.
  Let $\bar{\gamma}_1$ and $\bar{\gamma}_2$ be two elements of $\bar{\Gamma}$, and let $\Delta_1$, $\Delta_2$ be their oriented
axis in $\mathbb{H}^2$ as hyperbolic isometries of the hyperbolic plane. If $\Delta_1$ and $\Delta_2$ intersect transversely, or if they do not intersect but have the same direction, then $c(\bar{\gamma}_1, \bar{\gamma}_2)=0$.
\end{lemma}

\begin{proof}
By ``same direction" we mean that the attracting fixed points of $\bar{\gamma_1}, \bar{\gamma_2}$ do not separate in $\mathbb{R}P^1$ the repelling fixed points of these elements.

Since $\bar{\Gamma}$ is realized as a uniform lattice in PSL$(2, \mathbb{R})$, the  group $\widetilde{\Gamma}$ is realized as a lattice in $\widetilde{\operatorname{PSL}}(2, \mathbb{R})$.

Let $x_1$, $x_2$ be the attracting fixed point in $\mathbb{R}P^1$ of $\bar{\gamma}_1$,$\bar{\gamma}_2$, respectively, and
let $y_1$, $y_2$ be their repelling fixed points.

Let now $\gamma_i = \sigma_0(\bar{\gamma}_i)$. Let us fix one attracting fixed point $\tilde{x}_1$ for $\gamma_1$ in  $\widetilde{\mathbb{R}P}^1$.
Then the interval $]\tilde{x}_1, \delta\tilde{x}_1[$ contains one and only one repelling fixed point $\tilde{y}_1$ for $\gamma_1$, and only two fixed points
$\tilde{x}_2$, $\tilde{y}_2$ of $\gamma_2$, respectively attracting and repelling.

We have, for the cyclic order on $\mathbb{R}P^1$, only four possibilities:
\begin{enumerate}
\item  $x_1 < x_2 < y_1 < y_2$, or
\item  $x_1 < y_2 < y_1 < x_2$, or
\item  $x_1 < x_2 < y_2 < y_1$, or
\item  $x_1 < y_1 < y_2 < x_2$.
\end{enumerate}

In all these cases, there is a closed interval  in $\mathbb{R}P^1$ whose endpoints are $x_1$ and $x_2$,
and whose interior does not contain $y_1$ and $y_2$. This interval lifts to a closed interval $I_0$ in  $\widetilde{\mathbb{R}P}^1$
bounded by two attracting fixed points of $\gamma_1$, $\gamma_2$, and containing no other fixed point of $\gamma_1$, $\gamma_2$.
It follows that both $\gamma_1$ and
  $\gamma_2$ send the interval $I_0$ into itself. It follows that:
  $$\gamma_1\gamma_2(I_0) \subset I_0$$.

  Therefore, $\gamma_1\gamma_2$ admits a fixed point in $I_0$,
and $\gamma_1\gamma_2$ is $\sigma_0(\bar{\gamma}_1\bar{\gamma}_2)$. The Lemma is proved.
\end{proof}

\subsection{Anosov flows on the unit tangent bundle over a closed surface}\label{sub:tangentbundle}
A fundamental theorem %\footnote{This theorem extends to the case of Seifert manifolds - but there is no need in this paper to enter in such details.} 
is that the geodesic flow is essentially the unique Anosov flow on circle bundles - and our purpose here is to discuss
what exactly the term ``essentially'' means.
%the "essentiality" we are referring now.

\begin{theorem}[\cite{Gh}]\label{ghys}
  Let $\Phi^t$ be an Anosov flow on a closed orientable manifold $M$ admitting a circle fibration over a closed connected oriented surface $\Sigma$ of genus $g>1$. Then,
  there is a finite covering $p: M \to M_1(\Sigma)$ such that $\Phi^t$ is a reparametrization of the lifting in $M$ by $p$ of the Anosov geodesic flow of $\Sigma.$
\end{theorem}

%\begin{remark}\label{rk:ghys}
%  {\em
%  What actually E. Ghys proved is that, up to isotopy, the weak stable foliations of such an Anosov flow are transverse to the fibers, and that the
%projections in $\Sigma$ of the orbits of the flow are quasi-geodesics for any hyperbolic metric on $\Sigma$.
%}
%\end{remark}

In the rest of this Section, we study what are the isotopy classes of Anosov flows on the unit tangent bundle $M_1(\Sigma)$ itself. Let us introduce some notations:

-- Equip $p^0: M_1(\Sigma) \to \Sigma$ with a structure of principal $\mathbb{S}^1$-bundle (for example, select a metric on $\Sigma$ and then have the
action to be rotation by an angle in $\mathbb{S}^1$ $-$ this uses
that $\Sigma$ is oriented).

-- Let $\mathfrak{I}_1$ be the set of isotopy classes of Anosov flows on $M_1(\Sigma)$.

-- The fundamental group of $\Sigma$ has a presentation:
\begin{equation}\label{hh}
  \bar{\Gamma} = \langle \bar{a}_i, \bar{b}_i \; (i=1, ..., g) \: | \; [\bar{a}_1, \bar{b}_1][\bar{a}_2, \bar{b}_2] ... [\bar{a}_g, \bar{b}_g] = 1 \rangle
\end{equation}

-- The map $p^0: M_1(\Sigma) \to \Sigma$ is a fibration by circles of Euler class $2g-2.$ In particular, the fundamental group $\widetilde{\Gamma}$ of
$M_1(\Sigma)$ has a presentation:
\begin{equation}\label{hhh}
 \widetilde{\Gamma} = \langle h, a_i, b_i \; (i=1, ..., g) \: | \; [a_1, b_1][a_2, b_2] ... [a_g, b_g] = h^{2g-2} \rangle
\end{equation}
where the projections of the $a_i$'s and $b_i$'s are generators $\bar{a}_i$ and $\bar{b}_i$ of $\bar{\Gamma}$, and $h$ is represented by
the oriented fibers of $p$.

-- Let $p^0_\ast: \widetilde{\Gamma} \to \bar{\Gamma}$ be the map induced as the fundamental group level
by $p^0: M_1(\Sigma) \to \Sigma$.

%-- Let $\bar{\Gamma}$ be the fundamental group $\pi_1(\Sigma)$, and $\widetilde{\Gamma}$ the fundamental group of $M_1(\Sigma)$.

-- Let $\operatorname{Mod}(\Sigma)$ be the mapping class group of $\Sigma$, i.e. the group of orientation-preserving
homeomorphisms of $\Sigma$ up to isotopy.

-- Let $\operatorname{Mod}^\pm(\Sigma)$ be the extended mapping class group of $\Sigma$, i.e. the group of %orientation-preserving
homeomorphisms of $\Sigma$ up to isotopy.

-- Let $\operatorname{Mod}(M_1(\Sigma))$ be the extended mapping class group of $M_1(\Sigma)$, i.e. the group of
homeomorphisms of $M_1(\Sigma)$ up to isotopy. We will see in Remark \ref{rk:Soriented} that homeomorphisms of
 $M_1(\Sigma)$  are all orientation-preserving.

%-- Let $\operatorname{Mod}(M_1(\Sigma))$ be the extended mapping class group of $M_1(\Sigma)$, i.e. the group of %orientation-preserving homeomorphisms of $M_1(\Sigma)$ up to isotopy.

-- Note that ker $p^0_*$ is generated by $h$, and of course we require $\bar{a}_i = p^0_*(a_i)$ and $\bar{b}_i = p^0_*(b_i)$.

\begin{remark}\label{rk:isotopyhomotopy}
  {\em
  All the manifolds we consider are surfaces or irreducible Haken $3$-manifolds. Therefore, homeomorphisms are isotopic if and only if they are
  homotopically equivalent (\cite{waldlarge}).
  }
\end{remark}

\begin{remark}
{\em
In this context, the homeomorphism of $\oo$ induced by $h$ is $\zeta^2$, where $\zeta$ is the homeomorphism defined in Remark \ref{rk.TT}.
Therefore, in this context, Corollary \ref{cor:powerT2} becomes:

}

\begin{corollary}\label{cor:powerT3}
Let $F: M_1(\Sigma) \to M_1(\Sigma)$ be an orbital equivalence between the geodesic flow $\Phi_0$ and itself. Assume that $F$ is  isotopic to the identity. Then, $F$ is an isotopy along the orbits of $\Phi_0$.
\end{corollary}

This is because $\zeta^2$ sends any orbit to a
deck translate of it.

\end{remark}

 According to Baer-Dehn-Nielsen Theorem \cite{farbmargalit}, $\operatorname{Mod}^\pm(\Sigma)$ is isomorphic to Out$(\bar{\Gamma})$, i.e. the quotient
of the group Aut$(\bar{\Gamma})$ of automorphisms of $\bar{\Gamma}$ by the normal subgroup comprising inner automorphisms.
The mapping class group $\operatorname{Mod}(\Sigma)$ is isomorphic to Out$^+(\bar{\Gamma})$, i.e. the quotient by inner automorphisms of the group
of automorphism of $\bar{\Gamma}$ preserving the fundamental class,
see Theorem 8.1 of \cite{farbmargalit}.

Since $\Sigma$ has higher genus, every homeomorphism of $M_1(\Sigma)$ is isotopic to a homeomorphism preserving the fibers of $p^0: M_1(\Sigma) \to \Sigma$.
Therefore, there is a well-defined exact sequence:

\begin{equation}\label{eq:Mod1}
 1 \to K_1 \to \operatorname{Mod}(M_1(\Sigma)) \to \operatorname{Mod}^\pm(\Sigma) \to 1
\end{equation}

\noindent
where $K_1$ is the subgroup comprising isotopy classes of homeomorphisms preserving every fiber of $p^0: M_1(\Sigma) \to \Sigma$.

Let us fix a principal $\mathbb{S}^1$-bundle structure on $p^0: M_1(\Sigma) \to \Sigma$. For every $f$ in $K_1$, and every $x$ in $\Sigma$,
the restriction of $f$ to the fiber over $x$ is a homeomorphism $f_x$, and the principal $\mathbb{S}^1$-bundle structure provides
an identification of $f_x$ as an element of Homeo$(\mathbb{S}^1)$, well-defined up to conjugation by rotation.

Observe that every $f_x$ preserves the orientation of the fiber. Indeed, if it was not the case, all $f_x$ would be orientation reversing,
and would admit $2$ fixed points in every fiber. It would provide a section (maybe $2$-multivalued) of $p^0$, that is, a contradiction.

Therefore, every $f_x$ lies in Homeo$^+(\mathbb{S}^1)$.
In \cite[Proposition $4.2$]{Ghacting} E. Ghys defined a continuous retraction $\mathcal R$ from Homeo$^+(\mathbb{S}^1)$ into $\mathbb{S}^1$,
constant along classes of conjugation by rotation
as follows: any element $f$ of Homeo$^+(\mathbb{S}^1)$ lifts as a homeomorphism $\tilde{f}$ of the real line $\mathbb{R}$ commuting with the translation
$y \mapsto y+1$. Then $\mathcal R(f)$ is simply the integral
$\int_{0}^{1}(\tilde{f}(y)-y)dx \ {\rm mod} \  1$. This map is a homotopy equivalence. Observe that $\mathcal R(f)$ 
commutes with the composition by rotations, hence invariant by conjugation by a rotation. It therefore provides a well-defined map
$x \mapsto \mathcal R(f_x)$ that is continuous in $x$.
The map $f$ is homotopically equivalent to the homeomorphism of $M$ that acts by rotation of angle $\mathcal R(f_x)$
in each fiber $[p^0]^{-1}(x)$, and
since in closed Haken $3$-manifolds homotopically equivalent homeomorphisms are isotopic (\cite{waldlarge}), this homotopic equivalence is homotopic
to an isotopy.

It follows that $K_1$ is the group of homotopy 
classes of maps from $\Sigma$ into $\mathbb{S}^1$, which
is notoriously isomorphic to H$^1(\Sigma, \mathbb{Z}).$ The exact sequence \eqref{eq:Mod1} becomes:

\begin{equation}\label{eq:Mod2}
 1 \to \operatorname{H}^1(\Sigma, \mathbb{Z}) \to \operatorname{Mod}(M_1(\Sigma)) \to \operatorname{Mod}^\pm(\Sigma) \to 1
\end{equation}

\begin{remark}\label{rk:homeofibre}
{\em
The way $\operatorname{H}^1(\Sigma, \mathbb{Z})$ acts on $M_1(\Sigma)$ can described as follows. For every $\alpha$ in $\operatorname{H}^1(\Sigma, \mathbb{Z})$,
let $\omega$ be a closed $1$-form representing $\alpha$. Since $\alpha$ lies in $\operatorname{H}^1(\Sigma, \mathbb{Z})$, the periods $\int_c \omega$ for
closed loops $c$ are all integers. Select a point base $x_0$ in $\Sigma$, and for every $x$ in $\Sigma$, let $R(x)$ be the element
$\int_c \omega$ modulo $\mathbb{Z}$ of $\mathbb{R}/\mathbb{Z} \approx \mathbb{S}^1$, where $c$ is any path from $x_0$ to $x$.
Then let $f_\omega: M_1(\Sigma) \to M_1(\Sigma)$
be the map rotating every fiber $p^{-1}(x)$ by $R(x)$. The isotopy class of $f_\omega$ only depends on the cohomology class of $\omega$, and represents the element
$\alpha \in \operatorname{H}^1(\Sigma, \mathbb{Z}) \subset \operatorname{Mod}(M_1(\Sigma))$.

The induced action of $\operatorname{H}^1(\Sigma, \mathbb{Z})$ on $\widetilde{\Gamma}$ is described as follows:
the action of $\alpha \in
\operatorname{H}^1(\Sigma,\mathbb{Z})$ on
$\widetilde {\Gamma}$  is simply the map sending every $\gamma$
of $\widetilde{\Gamma}$ to
$\gamma h^{\alpha(\bar{\gamma})}$,  where
$\bar{\gamma} = p^0_{\ast}(\gamma)$.
Here we use the canonical identification
$\operatorname{H}^1(\Sigma, \mathbb{Z}) \approx \operatorname{H}^1(\bar{\Gamma}, \mathbb{Z})$.

A particular interesting case is the case of \emph{vertical Dehn twists over a simple closed curve}: Let $c$ be a simple closed oriented curve in $\Sigma$, and let
$U$ be a small collar neighborhood of $c$. The open domain $p^{-1}(U)$ is a tubular neighborhood of the torus $T := p^{-1}(c)$. Let $\omega_c$
be a closed $1$-form in $\Sigma$, with support in $U$, and representing the cohomology class dual to the homology class of $c$: for every loop $c'$,
the integral $\int_{c'} \omega_c$ is the algebraic intersection number between $c$ and $c'$. Then, the map $f_{\omega_c}$ is a vertical Dehn twist:
it is the identity map outside $p^{-1}(U)$, and inside $U$, it rotates the fibers, more and more when we go from the left to the right of $c$, so that
it adds to every homotopy class of curves crossing (positively) $c$ a $h$-component.

Since cohomology classes dual to simple closed curves generate $\operatorname{H}^1(\Sigma, \mathbb{Z})$, vertical Dehn twists generate
the kernel of the projection of $\operatorname{Mod}(M_1(\Sigma))$ onto
$\operatorname{Mod}^{\pm}(\Sigma)$. Actually, vertical Dehn twists over the generators
$\bar{a}_i$ and $\bar{b}_i$ of $\bar{\Gamma}$ are enough to generate this kernel.
}
\end{remark}

\begin{remark}\label{rk:mod0}
{\em
 The sequence \eqref{eq:Mod2} is split. Indeed, consider an element of $\operatorname{Mod}^\pm(\Sigma)$, represented by a homeomorphism $f$. Let
 $\tilde{f}$ be a lift of $f$ in $\mathbb{H}^2 \approx \widetilde{\Sigma}$. This lift is well-defined up to composition by
 an element of $\bar{\Gamma}$. It extends to the conformal boundary
 $\partial\mathbb{H}^2 \approx \mathbb{R}P^1$ as a homeomorphism $\partial f$.

Let us first consider the case where $f$ preserves the orientation, then its action on $\mathbb{R}P^1$ preserves the cyclic order,
and there induces an action on the space $X$ of triples $x < y < z$ introduced in Section \ref{sub.triple}.
The induced action on the quotient of $X$ by $\bar{\Gamma}$
 does not depend on the choice of the lift $\tilde{f}$. Since this quotient is homeomorphic to $M_1(\Sigma)$, this process defines a morphism
 from $\operatorname{Mod}(\Sigma)$ into $\operatorname{Mod}(M_1(\Sigma))$.
 We denote by $\operatorname{Mod}(\Sigma)^o$ the image of this morphism. Observe that as it is defined, $\operatorname{Mod}(\Sigma)^o$ is a set
of isotopy classes of homeomorphisms of $M_1(\Sigma)$, but actually we have realized it as a subgroup of genuine homeomorphisms of $M_1(\Sigma)$, in other words,
we have selected an element in each isotopy class in a coherent way with respect to the group structure.

This construction does not apply when $f$ is orientation reversing, since then for every element $(x, y, z)$ of $X$ the triple $(\partial f(x), \partial f(y), \partial f(z))$
does not belong to $X$ since $\partial f(x) > \partial f(y) > \partial f(z)$. But it suffices to define the action of $\partial f$ on $X$ as given by:
$$\partial f.(x,y,z) = (\partial f(y), \partial f(x), \partial f(z))$$
Once again, the induced action on $\bar{\Gamma}\backslash X$ does not depend on the choice of the lift $\tilde{f}$, and this process
 defines a group homomorphism
 from $\operatorname{Mod}^\pm(\Sigma)$ into $\operatorname{Mod}(M_1(\Sigma))$,  which is a section of the projection $\operatorname{Mod}(M_1(\Sigma)) \to \operatorname{Mod}^\pm(\Sigma)$.  We denote by $\operatorname{Mod}^\pm(\Sigma)^o$ the image of this section.

Again observe that as it is defined, $\operatorname{Mod}^\pm(\Sigma)^o$ is a set
of isotopy classes of homeomorphisms of $M_1(\Sigma)$, but actually we have realized it as a subgroup of genuine homeomorphisms of $M_1(\Sigma)$, in other words,
we have selected an element in each isotopy class in a coherent way with respect to the group structure.
}
\end{remark}

\begin{remark}\label{rk:Soriented}
  {\em
Observe that the homeomorphisms defined in Remark \ref{rk:mod0} preserve the orientation of $X$, in the case where $f$ is orientation-preserving,
but also in the case where it is orientation reversing. Indeed, in the last case, the homeomorphism is the restriction to $X$ of the composition of the orientation reversing map $(x,y,z) \mapsto (\partial f(x), \partial f(y), \partial f(z))$ and the orientation reversing map $(x,y,z) \mapsto (y,x,z)$. On the other hand, the description in Remark \ref{rk:homeofibre} makes it clear that elements
of $K_1$ are also orientation-preserving. 

In view of the sequence \eqref{eq:Mod2}, it follows that every element of $\operatorname{Mod}(M_1(\Sigma))$ is orientation-preserving.
  }
\end{remark}

\begin{remark}\label{rk:BaerDehnNielsen}
{\em
This construction is closely related to the Baer-Dehn-Nielsen Theorem mentioned above. Actually, the circle $\mathbb{S}^1$ is naturally identified with
the Gromov boundary $\partial\bar{\Gamma}$; hence the group of automorphisms Aut$(\bar{\Gamma})$ acts on $\mathbb{S}^1$, and therefore on the space $X$ of distinct triple of points in $\mathbb{S}^1$ satisfying $x < y < z$, if one performs as in Remark \ref{rk:mod0} in the orientation reversing case.

This action induces an action of Out$(\bar{\Gamma})$ on the quotient of $X$ by $\bar{\Gamma}$, i.e. a morphism
from Out$(\bar{\Gamma})$ into $\operatorname{Mod}(M_1(\Sigma))$. Composing this morphism with the projection induces a morphism Out$(\bar{\Gamma}) \to$ $\operatorname{Mod}^\pm(\Sigma)$ which
is precisely the Baer-Dehn-Nielsen isomorphism. Observe that there is an obvious inverse of this map: the one associating to an element of $\operatorname{Mod}^\pm(\Sigma)$
its induced action on the fundamental group $\bar{\Gamma}$.
}
\end{remark}

Recall that $\mathfrak{I}_1$ is the set of isotopy classes of Anosov flows on $M_1(\Sigma)$. The modular group $\operatorname{Mod}(M_1(\Sigma))$ acts clearly on $\mathfrak{I}_1$
since every element of $\operatorname{Mod}(M_1(\Sigma))$ admits a smooth representative (but there is no way to choose simultaneously such representatives
for every element of $\operatorname{Mod}(M_1(\Sigma))$, see \cite{Souto}).

\begin{proposition}\label{pro:anosovM1}
  The action of $\operatorname{Mod}(M_1(\Sigma))$ on $\mathfrak{I}_1$ is transitive, and the stabilizer of the isotopy class of the geodesic flow $\Phi_0$ is the subgroup  $\operatorname{Mod}^\pm(\Sigma)^o$ defined in Remark \ref{rk:mod0}.
\end{proposition}

\begin{proof}
  The transitivity property of the action comes from Theorem \ref{ghys}:
this theorem gives us a finite cover $p: M_1(\Sigma) \mapsto M_1(\Sigma)$
satisfying the conclusion of Theorem \ref{ghys}. Using the Euler
class of the bundle $M_1(\Sigma)$ it follows that the cover is
of degree one, hence a homeomorphism, so $p$ is in $\operatorname{Mod}(M_1(\Sigma))$.
Transitivity follows.

We now analyze the stabilizer.
Denote by $[\Phi_0]$ the isotopy class of the geodesic flow on $M_1(\Sigma)$,
  and by $\Phi_0$ the representative of $[\Phi_0]$ induced by
an arbitrary but fixed hyperbolic metric $g_0$ on $\Sigma$.
We fix $\Phi_0$ throughout the proof.

  Let now $\eta$ be an element of $\operatorname{Mod}^\pm(\Sigma)$.
  Consider as in Remark \ref{rk:mod0} the action $\partial f$ on the circle $\partial\mathbb{H}^2 \approx \mathbb{R}P^1$ of
  some lift $\tilde{f}$ of a representative $f$ of $\eta$ (hence
$\eta = [f]$). Remark \ref{rk:mod0} is based on the observation that the transformation
 $$(x,y,z) \ \mapsto \ (\partial f(x), \partial f(y), \partial f(z))$$
\noindent
when $f$ is orientation-preserving, and
$$(x,y,z) \ \mapsto \ (\partial f(y), \partial f(x), \partial f(z))$$
\noindent
 when $f$ is orientation reversing,  is $\bar{\Gamma}$-equivariant on $X$. Therefore
this transformation induces
  a homeomorphism $F$  on $\bar{\Gamma}\backslash X \approx M_1(\Sigma)$. This homeomorphism $F$ preserves the \textit{non-oriented}  foliation $x = Cte$, $y = Cte$.
  As explained at the end of Section \ref{sub:geodesic}, this foliation is a representative of the \textit{non-oriented} foliation induced by the geodesic flow. Therefore,  $F_\ast\Phi_0$ is isotopic to the geodesic flow or to its inverse.
But the geodesic flow is isotopic to its own inverse,
see the observation just after Theorem \ref{thm:ba1}.
Hence, in all cases, $F_\ast\Phi_0$ is isotopic to the geodesic flow.
We conclude that
$\operatorname{Mod}^{\pm}(\Sigma)^o$ is contained in the stabilizer of $[\Phi_0]$.

  Since the sequence $(4)$ is split it follows that
the group $\operatorname{H}^1(\Sigma, \mathbb{Z}) \approx \operatorname{Mod}(M_1(\Sigma))/\operatorname{Mod}^\pm(\Sigma)^o$ acts
  transitively on $\mathfrak{I}_1$. In order to conclude, we just have to show that non-trivial elements of $\operatorname{H}^1(\Sigma, \mathbb{Z})$
  do not belong to the stabilizer of $[\Phi_0]$.

  Let $\alpha$ be an element of $\operatorname{H}^1(\Sigma, \mathbb{Z})$, and let $\bar{\gamma}$ be a non-trivial element of $\bar{\Gamma}$.
  %let us consider the action of $(p_\ast)^{-1}(\bar{\Gamma})$ on the leaf space $Q^s_0 \approx \widetilde{\mathbb{R}P}^1$ associated to the geodesic flow $\Phi_0^t$.
The action of $\bar{\Gamma}$ on $\mathbb{H}^2 \approx \widetilde{\Sigma}$ preserves a unique geodesic
$(xy)$ in $\mathbb{H}^2$, with $x$, $y \in \partial\mathbb{H}^2$, so that $\bar{\gamma}$ acts on this geodesic as a translation from $y$ towards $x$.
The vectors in $T^1\mathbb{H}^2$ tangent to this geodesic and pointing in the direction of $x$ form an orbit of the geodesic flow of
$\mathbb{H}^2$.

This geodesic lifts in the universal covering of $M_1(\Sigma)$ to infinitely many orbits of $\widetilde{\Phi}_0^t$ that are
permuted under the action of the fiber $h$, and there is one and only one element $\gamma = \sigma_0(\bar{\gamma})$ of
$(p^0_\ast)^{-1}(\bar{\gamma})$ preserving each of these lifted geodesics (see Remark \ref{rk:actionQ}): all other elements
of $(p^0_\ast)^{-1}(\bar{\Gamma})$ have the form
$\gamma h^k$ for some integer $k$, and when $k \neq 0$, $\gamma h^k$ preserves no orbit of $\widetilde{\Phi}_0^t$.

Now apply the element $\alpha \in \operatorname{H}^1(\Sigma, \mathbb{Z}) \subset \operatorname{Mod}(M_1(\Sigma))$ to the geodesic flow $\Phi^t_0$, and assume
that $\alpha_*\Phi^t_0$ is isotopic to $\Phi^t_0$.  The
homotopy class $\gamma$ is mapped under $\alpha$ to the homotopy class $\gamma h^{\alpha(\bar{\gamma})}$. Since
$\alpha_*\Phi^t_0$ is isotopic to $\Phi^t_0$, this homotopy class must preserve an orbit of $\widetilde{\Phi}_0^t$, hence
we must have $\alpha(\bar{\gamma})=0$. Therefore, since $\bar{\gamma}$ is arbitrary, $\alpha$ is trivial.
The proposition follows.
\end{proof}

In summary, there are infinitely many isotopy classes of Anosov flows on $M_1(\Sigma)$ parametrized by
$\operatorname{Mod}^\pm(M_1(\Sigma))/\operatorname{Mod}^\pm(\Sigma)^0 \approx \operatorname{H}^1(\Sigma, \mathbb{Z}) \approx \mathbb{Z}^{2g}$. It follows that isotopy classes can be
distinguished according to periodic orbits over the generators $\bar{a}_i$ and $\bar{b}_i$ of $\bar{\Gamma}$:
if two Anosov flows on $M_1(\Sigma)$ have the property
that the periodic orbits ``above the generators $\bar{a}_i$ and $\bar{b}_i$'', for all $i$, are freely
homotopic for the two flows, then the two flows are isotopic.

\subsection{Isotopy classes of Anosov flows on circle bundles over closed surfaces}\label{sub:along}

In all this subsection, $M$ is an oriented $3$-manifold, admitting a fibration by circles $p: M \to \Sigma$ over a closed, connected, oriented hyperbolic
surface $\Sigma$.  We will denote by $\Theta_0$ the foliation whose leaves are the fibers of $p.$ 
Observe that this foliation is oriented, is a foliation by circles, and  is unique up to isotopy. 

This section is devoted to the study of isotopy classes of Anosov flows on $M$. In the next subsection we will study orbit equivalence classes
of Anosov flows. In the next section we prove the main theorem of this paper (Theorem \ref{th:maincircle}).

%It is closely related
%to the study of isotopy classes and conjugacy classes of contact $1$-forms on $M$ in \cite{Giroux}. We will conclude this Section with a discussion
%on this fact.

According to Theorem \ref{ghys}, up to reparametrization, every Anosov flow on $M$ is the pull-back of the geodesic flow
$\Phi_0^t$ by a finite covering $q: M \to M_1(\Sigma)$. Moreover, $M_1(\Sigma)$ is precisely the unit tangent bundle over
the surface $\Sigma$ over which $M$ is
assumed to be fibering. It implies that the finite covering $q: M \to M_1(\Sigma)$ is
\emph{a covering along the fibers}, meaning that $p$ is isotopic to
the composition $p^0 \circ q$.
%This is because Seifert fiberings in $M$ are unique up to isotopy: both $p$ and $p^0 \circ q$ define a $\mathbb{S}^1$-bundle structure in $M$, hence a Seifert fibering.
In other words, the pull-back by $q$ of the fibration $p^0$ is
isotopic to the fibration induced by $p$. This again is because
the Seifert fibration in $M$ is unique up to isotopy.
Modifying $q$
by this isotopy (notice that this does not modify the
isotopy class of the Anosov foliation in $M$), one can always assume that the equality $p = p^0 \circ q$ holds.

In other words, a covering $N \to M_1(\Sigma)$ is not along the fibers
if the base surface of the Seifert fibering of $N$ is not homeomorphic
to $\Sigma$, or, alternatively, if the preimages of fibers of $M_1(\Sigma)$ by the covering map are not connected.

All these finite coverings $q: M \to M_1(\Sigma)$ have the same degree $n$:
it is the quotient $(2g-2)/|d|$, where $|d|$ is the absolute value of the Euler class $d$ of $p: M \to \Sigma$.

This remark makes also clear that all the finite covers $M_1(\Sigma)$ of  degree $n$ along fibers are all homeomorphic one to the others:
there are all circle bundles over $\Sigma$ of Euler class $(2g-2)/n$.

We propose now a useful alternative definition of finite coverings along the fibers. 
Every leaf of $\Theta_0$, i.e. every fiber of $p$, is an oriented circle. One can therefore define the homeomorphism $g_0$ mapping
every $x$ in $M$ to the next element of the leaf of $\Theta_0$ that
maps to the same point as $x$ by $p$, according to the orientation of this fiber. This homeomorphism has order $n$ and induces a free action of $\mathbb{Z}/n\mathbb{Z}$ on $M.$

More generally, let $\Theta$ be an oriented foliation of $M$ by circles, and let $g$ be a homeomorphism of $M$ of order $n$, preserving every leaf of $\Theta$ and such that the 
action of $\mathbb{Z}/n\mathbb{Z}$ on $M$ induced by $g$ is free. Then, we say that $g$ is \textit{$\Theta$-increasing} if:

-- either $n=1$ or $2$, 

-- or $n \geq 2$ and for every $x$ in $M$ the triple $(x, g(x), g^2(x))$ is positive for the cyclic order of the (oriented) leaf of $\Theta$ containing $x$.

\begin{lemma}\label{le:along}
There is an one-to-one correspondence between:
\begin{itemize}
\item Finite coverings $q: M \to M_1(\Sigma)$  of degree $n$ along the fibers,
\item The data of an oriented foliation $\Theta$ of $M$ by circles, a $\Theta$-increasing homeomorphism $g$ of $M$ of order $n$ without fixed points preserving every leaf of $\Theta$, and an
homeomorphism between $M_1(\Sigma)$ and the orbit space of the $\mathbb{Z}/n\mathbb{Z}$-action in $M$.
\end{itemize}
\end{lemma}

The proof of this Lemma is quite straightforward, once observed that the only foliations of $M$ by circles are the ones induced by circle fibrations of $M$ over $\Sigma$. 
This is because any foliation by circles in $M$ is a Seifert fibration
\cite{Ep}, then 
Seifert fibrations in $M$ are unique up to isotopy, and finally
there is a circle fibration of $M$ over $\Sigma$.
We just point out that given the free $\mathbb{Z}/n\mathbb{Z}$-action and the homeomorphism $\Psi$  between the orbit space $N$ 
of the $\mathbb{Z}/n\mathbb{Z}$-action and $M_1(\Sigma)$, the finite covering is simply $\Psi \circ \varrho$, where $\varrho: M \to N$ is the projection to the
quotient space.
The other direction is straightforward.

%Finally, one can assume that the fibration $p$ is the composition $p^0 \circ q_0$.

After all these preliminaries, we go back to the problem of the study of isotopy classes of Anosov flows on $M$.
Once again, according to Theorem \ref{ghys}, it leads to the study of the set $\mathfrak{Cov}^\pm_n$
of finite coverings of degree $n$ along the fibers up to the following equivalence relation: two such
covering maps $q_1: M \to M_1(\Sigma)$ and $q_2: M \to M_1(\Sigma)$ are
{\emph{isotopic}} if there exists a homotopically trivial
homeomorphism $F: M \to M$ such that $q_2 =  q_1 \circ F$. Indeed,
if $q_1$ and $q_2$ are isotopic in this sense, then the Anosov flows $q_1^\ast\Phi_0$ and $q_2^\ast\Phi_0$ are isotopic in $M$.

Any homotopically trivial homeomorphism of $M_1(\Sigma)$ lifts to a homotopically trivial homeomorphism of $M$, therefore
the mapping class group $\operatorname{Mod}(M_1(\Sigma))$ acts on $\mathfrak{Cov}^\pm_n$ by composition on the left.
The $\pm$ refers to covers preserving or reversing orientation.

\begin{proposition}\label{prop:ModCtransitive}
  The action of $\operatorname{Mod}(M_1(\Sigma))$ on $\mathfrak{Cov}^\pm_n$ is simply transitive.
\end{proposition}

\begin{proof}
First observe that this action is free. Indeed: let$f$ be a homeomorphism of $M_1(\Sigma)$ and $q_1: M \to M_1(\Sigma)$ be a finite covering such that there is a 
homotopically trivial homeomorphism $F$ of $M$ such that $f \circ q_1 = q_1 \circ F$. Then, $f$ must be homotopically trivial too, and therefore trivial in
 $\operatorname{Mod}(M_1(\Sigma))$.

Let now $q_0: M \to M_1(\Sigma)$ and $q: M \to M_1(\Sigma)$ be two covering maps of degree $n$ along the fibers. We
want to show that $q$ is isotopic to a finite covering of the form $f \circ q_0$ for some homeomorphism $f$ of $M_1(\Sigma)$.

According to Lemma \ref{le:along}, one can interpret $q_0$ and $q$ as the data of two foliations by circles $\Theta_0$ and $\Theta$ on $M$, two free actions 
of $\mathbb{Z}/n\mathbb{Z}$ induced by increasing homeomorphisms $g_0$ and $g$ on $M$, the first preserving $\Theta_0$ leafwise and the second preserving $\Theta$ leafwise, and identifications $f_0: N_0 \to M_1(\Sigma)$ and
$f: N \to M_1(\Sigma)$, where $N_0$ is the orbit space of the first $\mathbb{Z}/n\mathbb{Z}$-action, and $N$ the orbit space of the second $\mathbb{Z}/n\mathbb{Z}$-action.

Up to isotopy, one can assume that the circle foliations $\Theta_0$ and $\Theta$ coincide. Then, in every leaf of $\Theta = \Theta_0$, the two $\mathbb{Z}/n\mathbb{Z}$-actions
are both conjugated (through a homeomorphism isotopic to the identity) to the same action by rotations, meaning that they are conjugated one to the other. These conjugacies
in all the leaves of $\Theta$ can be combined to some continuous conjugacy $F$ between the two $\mathbb{Z}/n\mathbb{Z}$-actions preserving every leaf of $\Theta_0.$ 
This conjugacy is not necessarily isotopic to the identity.

Let us fix a principal $\mathbb{R}/\mathbb{Z}$-structure on  $p^0: M_1(\Sigma) \to \Sigma.$ The pull-back of principal $\mathbb{R}/\mathbb{Z}$-structure by the finite covering $q_0$ is  
a principal $\mathbb{R}/n\mathbb{Z}$-structure on $p: M \to \Sigma$. Then, for every $x$ in $M$ and every $a \in \mathbb{R}$ the image by $q_0$ of $a.x$ is equal to $a.q_0(x).$
There is a map $f: \Sigma \to$ Homeo$^+(\mathbb{R}/n\mathbb{Z})$ such that, for every $x$ in $M$, the image $F(x)$ is the image of $x$ under $f(p(x)) \in$ Homeo$^+(\mathbb{R}/n\mathbb{Z})$. 

We apply Ghys' argument as we did just after Corollary \ref{cor:powerT3}: up to isotopy on $F$, one can assume that every $f(x)$ for $x$ in $\Sigma$ is a rotation. In other words,
for every $x$ in $M$ we have $F(x) = f(p(x)).x$, where $f(p(x))$ in an element of $\mathbb{R}/n\mathbb{Z}$ $-$ here we are now identifying the
rotation number with the rotation homeomorphism associated with that
rotation number. Let $[f(p(x)]$ be the projection of $f(p(x)) \in \mathbb{R}/n\mathbb{Z}$ in $\mathbb{R}/\mathbb{Z}$.  Then, the map $x \mapsto [f(p(x)]$ defines a vertical Dehn twist $f_\omega$ as described in Remark \ref{rk:homeofibre}. 

Therefore, up to replacing $q_0$ by $f_\omega \circ q_0$ where $f_\omega \in K_1 \subset \operatorname{Mod}(M_1(\Sigma))$, one can assume that both $\mathbb{Z}/n\mathbb{Z}$-actions coincide.

Therefore, the orbit
spaces $N$ and $N_0$ do coincide, and the projection maps $\varrho_0: M \to N_0$ and $\varrho: M \to N$ as well.

The last step is to observe that the homeomorphisms $f_0: N_0 \to M_1(\Sigma)$ and $f: N_0 \to M_1(\Sigma)$ do not necessarily coincide, even if $N = N_0$. 
However, it means that $q$ coincides with $g \circ q_0$, where $g = f \circ f^{-1}_0$, meaning as required that the isotopy class of $q$ is the image of the isotopy class of
$q_0$ under $[g]$. The Proposition is proved.
\end{proof}

\begin{remark}\label{rk:sanspm}
  {\em
Since every element of  $\operatorname{Mod}^\pm(M_1(\Sigma)) = \operatorname{Mod}(M_1(\Sigma))$ is orientation-preserving (Remark \ref{rk:Soriented}) it follows from Proposition \ref{prop:ModCtransitive} that every finite covering over $M_1(\Sigma)$ is orientation-preserving, i.e. that $\mathfrak{Cov}^\pm_n = \mathfrak{Cov}_n$.
  }
\end{remark}

Let $\mathfrak{I}_n$ be the space of isotopy classes of Anosov flows in $M$. According to  Theorem \ref{ghys}, every Anosov foliation on $M$ is the pull-back $q^\ast\Phi_0$ of the geodesic foliation $\Phi_0$  by a finite covering $q: M \to M_1(\Sigma)$. It means that the map
$\#: \mathfrak{Cov}_n \to \mathfrak{I}_n$ associating to (the isotopy class of) a finite covering  $q: M \to M_1(\Sigma)$ the (the isotopy class of the) pull-back $q^\ast\Phi_0$ of the geodesic flow is surjective.

\begin{corollary}\label{cor:ModTransitive}
The fibers of the map $\#: \mathfrak{Cov}_n \to \mathfrak{I}_n$
are precisely the $\operatorname{Mod}^\pm(\Sigma)^0$-orbits.
In particular, there is a one-to-one correspondence between $ \mathfrak{I}_n$ and  $K_1 \approx \operatorname{H}^1(\bar{\Gamma}, \mathbb{Z})$.
\end{corollary}

\begin{proof}
We use Proposition \ref{prop:ModCtransitive}. 
Since  $\operatorname{Mod}^\pm(\Sigma)^0$ is the stabilizer of $[\Phi_0]$, the map  $\#: \mathfrak{Cov}_n \to \mathfrak{I}_n$ is constant along  $\operatorname{Mod}^\pm(\Sigma)^0$-orbits.
Since the exact sequence  \eqref{eq:Mod2} is split, it follows that the restriction of $\#$ to any $K_1$-orbit is still surjective. The Corollary will
be proved if we show that this restriction to some $K_1$-orbit is injective.
We will show it by using the same argument as in the end of the proof of Proposition \ref{pro:anosovM1}.

More precisely: let us consider the $K_1$-orbit of the preferred finite covering $q_0: M \to M_1(\Sigma)$ we have selected.
%in the proof of Proposition \ref{prop:ModCtransitive} and its
Let $\alpha$ be a non trivial  element of $K_1 \approx \operatorname{H}^1(\bar{\Gamma}, \mathbb{Z})$.
Let $q$ be the finite covering obtained by composing the vertical twist associated to $\alpha$ and $q_0$. Assume that $q_0^\ast\Phi_0$ and $q^\ast\Phi_0$ are isotopic in $M$.
But since $q = \alpha \circ q_0$ then $q^\ast = q_0^\ast \circ \alpha^\ast$.

Now consider $\Phi_0$ and $\Phi_1 = \alpha^\ast(\Phi_0)$.
Since $\alpha$ is in $K_1$ it follows that $\Phi_0$ and $\Phi_1$ are not isotopic in $M_1(\Sigma)$. In fact,
according to the discussion at the end of Section \ref{sub:tangentbundle},
there is some simple closed geodesic $\theta$ in $\Sigma$ representing
some non-trivial element $\bar{\gamma}$ of $\bar{\Gamma}$ such
that the periodic orbits $\theta_0, \theta_1$
of $\Phi_0, \Phi_1$ respectively above $\theta$ are not isotopic
in $M_1(\Sigma)$. It follows that the periodic orbits
of $q^\ast_0 \Phi_0$ and $q^\ast_0 \Phi_1$ above $\theta$ are
not isotopic.
But $$q^\ast_0 \Phi_1 = (\alpha \circ q_0)^\ast \Phi_0 = q^\ast \Phi_0.$$
\noindent
By hypothesis $q^\ast_0 \Phi_0$ and $q^\ast \Phi_0$ are
isotopic flows. Hence these periodic
orbits should all be isotopic to each other, since the periodic
orbits in the the torus above $\theta$ are all isotopic
to each other. This is a contradiction.
%Let $\theta$ be a periodic orbit of $q_0^\ast\Phi_0$: it represents (up to inner automorphisms) an element $\gamma$ of $\Gamma$.
%As explained in Remark \ref{rk:homeofibre}, the lift by $q$ of the periodic orbit of $\Phi_0$ below
%$\theta$ represents in $\Gamma$ the (free homotopy class of the) element $\gamma h^\alpha(\gamma)$. Then, $\gamma$ and $\gamma h^\alpha(\gamma)$ are two elements in $\Gamma$ both preserving an orbit of $\widetilde{\Phi}_0$, and commuting one to the other. It implies that
%they must be equal, hence $\alpha(\gamma) = 0$. Since it must hold for every periodic orbit, it follows that $\alpha$ is trivial in $ \operatorname{H}^1(\bar{\Gamma}, \mathbb{Z})$.

The corollary is proved.
\end{proof}

\subsection{Anosov flows on circle bundles over closed surfaces up to orbital equivalence}\label{sub:G_n}

In this subsection, we still denote by $M$ a fixed closed orientable circle bundle, which admits a finite covering of degree $n$ along the fibers over $M_1(\Sigma).$ 
We will study orbital equivalences of Anosovflows on $M$, i.e. finite coverings of geodesic flows. For this purpose, we change our point of view and consider
the geodesic flow as an algebraic flow (cf. Section \ref{sub.algebraic}). Hence finite coverings can be described as quotients of $\widetilde{\operatorname{PSL}}(2, \mathbb{R})$ by a subgroup $\Gamma$ of index $n$ in the lattice $\widetilde{\Gamma}$.

\begin{define}\label{def:GGn}
Let  $\mathfrak{G}_n$ be the set of subgroups $\Gamma$
of $\widetilde{\Gamma}$ such that:

-- the restriction of $p^0_\ast: \widetilde{\Gamma} \to \bar{\Gamma}$
to $\Gamma$ is surjective,

-- $\Gamma$ has finite index $n$ in $\widetilde{\Gamma}$
\end{define}

In other words,  $\mathfrak{G}_n$ is the set of subgroups of $\widetilde{\Gamma}$ corresponding to finite coverings of degree $n$ along the fibers. Observe that for any
element $\Gamma$ of $\mathfrak{G}_n$, 
the center $Z(\Gamma)$ is the intersection between $\Gamma$ and the center of  $\widetilde{\Gamma}$. More precisely, if $h$ denotes a generator of the center of $\widetilde{\Gamma}$
(corresponding to fibers of $M_1(\Sigma)$), the center is generated by $h^n.$ Therefore, the first condition in Definition \ref{def:GGn} means that the quotient group 
${\Gamma}/Z({\Gamma})$ is isomorphic to $\bar{\Gamma}.$

According to Ghys's Theorem, any Anosov flow on $M$ is orbitally equivalent to the flow on ${\Gamma}\backslash\widetilde{\operatorname{PSL}}(2, \mathbb{R})$ induced 
by the right action of the diagonal group $D$ for some element $\Gamma$ of $\mathfrak{G}_n$ (see subsection \ref{sub.algebraic}). We will denote by $\Phi_\Gamma$ this flow.

The key observation we want to start with is that $\operatorname{Mod}(M_1(\Sigma))$ acts naturally on  $\mathfrak{G}_n$. Indeed: let $[F]$ be a element of  $\operatorname{Mod}(M_1(\Sigma))$, represented by some homeomorphism $F$ on $M_1(\Sigma)$. Let $F_\ast$ be the action on $\pi_1(M_1(\Sigma)) \approx \widetilde{\Gamma}$. This automorphism of $\widetilde{\Gamma}$ is defined up to inner automorphism, but we claim that for any element $\Gamma$ of  $\mathfrak{G}_n$,
the subgroup $F_\ast(\Gamma)$ does not depend on this choice.

Indeed: let $\tilde{\gamma}$ be any element of $\widetilde{\Gamma}$. Then, according to the first property in Definition \ref{def:GGn},  there is an element $\gamma$ of $\Gamma$ and an integer $k$ such that $\tilde{\gamma} = \gamma h^k$.  Then, since $h$ lies in the center of $\widetilde{\Gamma}$, the conjugate
of $\Gamma$ under $\tilde{\gamma}$ is equal to its conjugate by $\gamma$, and therefore itself since $\gamma$ lies in $\Gamma$. The claim is proved.

Moreover, since $F_\ast$ must preserves the kernel of $p^0_\ast$ (because it is the center of $\widetilde{\Gamma}$), the image $F_\ast(\Gamma)$ is still an element of  $\mathfrak{G}_n$. We have proved our first key observation.

Therefore, the subgroup Mod$^\pm(\Sigma)^0$ acts naturally on $\mathfrak{G}_n$.

More precisely, let $\varphi$ be an automorphism of $\bar{\Gamma},$ representing an element $[\varphi]$  of Out$(\bar{\Gamma}) \approx$ Mod$^\pm(\Sigma),$ therefore an element of 
Mod$^\pm(\Sigma)^0$. As explained in Remarks \ref{rk:mod0}, \ref{rk:Soriented}, $\varphi$ induces an homeomorphism $F$  in $M_1(\Sigma)$, which is an self orbital equivalence of the geodesic flow. Therefore, $\varphi$ maps every element $\Gamma$ of  $\mathfrak{G}_n$ to the subgroup $F_\ast(\Gamma).$ Now, 
for every element $\gamma$ of $\Gamma$ let $f_\ast(\gamma)$ be the image under $p^0_\ast$ of $F_\ast(\gamma)$. If $\gamma'$ is another element of $\Gamma$ admitting the same projection in $\bar{\Gamma}$ than $\gamma$, there is some integer $k$ such that $\gamma' = h^k \gamma.$ Then:
$$F_\ast(\gamma') = F_\ast(\gamma) F_\ast(h)^k = F_\ast(\gamma) h^{\pm k}.$$
Therefore, $f_\ast(\gamma') = f_\ast(\gamma),$ and actually since the projection by $p^0_\ast$ restricted 
to $\Gamma$ and $F_\ast(\Gamma)$ as well is surjective, $f_\ast$ defines an automorphism of $\bar{\Gamma}$ into itself, well-defined up to inner automorphims. Actually, 
it follows from our construction that this automorphism, up to some inner automorphism, coincides with the initial automorphism $\varphi.$

We can now state the main result of this subsection:

\begin{proposition}\label{pro:Fgroup}
Let $\Gamma$, $\Gamma'$ be two elements of $\mathfrak{G}_n$. Then, $\Phi_{\Gamma}$ and $\Phi_{\Gamma'}$ are orbitally equivalent if and only if $\Gamma$ and $\Gamma'$  lie in the same orbit under the action of $\operatorname{Mod}^\pm(\Sigma)^0$ on $\mathfrak{G}_n$.
\end{proposition}

\begin{proof}
Since the homeomorphisms of $M_1(\Sigma)$ send a  geodesic flow foliation
to a geodesic flow foliation, it is quite clear that
if $\Gamma$ and $\Gamma'$ lie on the same $\operatorname{Mod}^\pm(\Sigma)^0$-orbit, the associated Anosov flows are orbitally equivalent - there is a homeomorphism
between $\Gamma\backslash\mi$ and $\Gamma'\backslash\mi$ mapping the first finite covering of the geodesic flow to a flow isotopic to
the second finite covering of the geodesic flow.

Assume now that the Anosov flows associated to $\Gamma$ and $\Gamma'$ are orbitally equivalent.
It means that there is an orbital equivalence between $\widetilde{\Phi}$
(which is the lift to the universal cover of the geodesic flow
on the surface) and itself. This  induces
a homeomorphism $F: \oo \to \oo$
on the orbit space.

Furthermore, this orbital equivalence preserves the stable and unstable foliations, and therefore induces homeomorphisms $F^s$ and $F^u$ of the leaf spaces $Q^s$ and $Q^u$, respectively.

The orbital equivalence induces a morphism $F_\ast: \Gamma \to \Gamma'$, such that, for every element $\gamma$ of $\Gamma$:
\begin{eqnarray}\label{eq:F}
F \circ \gamma & =  & F_\ast(\gamma) \circ F
\end{eqnarray}

Here we are abusing notation and also denoting by $F$ a 
lift of the 
orbital equivalence in question.
Similarly:
\begin{eqnarray*}
F^s \circ \gamma & =  & F_\ast(\gamma) \circ F^s\\
F^u \circ \gamma & =  & F_\ast(\gamma) \circ F^u
\end{eqnarray*}

According to the discussion just before the statement of Proposition \ref{pro:Fgroup}, if we replace
$\Gamma'$ by some element of its  $\operatorname{Mod}^\pm(\Sigma)^0$-orbit, one can
assume that $F_\ast$ induces a trivial element of Out$(\bar{\Gamma}).$ Therefore, there exists an element $\gamma_0$ of $\Gamma$ and a map $\alpha: \Gamma \to \mathbb Z$ such that:

$$\forall \gamma \in \Gamma \;\; F_\ast(\gamma) = h^{\alpha(\gamma)}\gamma_0\gamma\gamma_0^{-1}.$$

Replace $F$ by $F \circ \gamma_0^{-1}$, which is still an orbital equivalence between $\widetilde{\Phi}$ and itself. Then:
$$\forall \gamma \in \Gamma \;\; F_\ast(\gamma) = h^{\alpha(\gamma)}\gamma.$$

It follows that $\alpha: \Gamma \to \mathbb Z$ is a group homomorphism.

Recall (Section \ref{sub.W}) that  in the case of the geodesic flow there is an identification between
$\oo$ and the region $\Omega_0$ of $\widetilde{\mathbb{R}P}^1 \times \widetilde{\mathbb{R}P}^1$ between the
graphs of the identity map and the graph of $\delta$,
where $\delta$ generates the center of
$\widetilde{\operatorname{PSL}}(2,\mathbb{R})$.
Then, there is some homeomorphism $\bar{F}: \widetilde{\mathbb{R}P}^1  \to \widetilde{\mathbb{R}P}^1 $ (corresponding to $F^s$ or $F^u$ when $Q^s$ or $Q^u$ is identified with $\widetilde{\mathbb{R}P}^1 $) so
that $F$ corresponds to the map $(x,y) \mapsto (\bar{F}(x), \bar{F}(y))$.
Here we are using the identification of $Q^u$ with $Q^s$ via
the map $\beta^{-1}$, see Remark \ref{rk.TT}.
Since this diagonal map must preserve $\Omega_0$, it follows
that for every $x$ in $\widetilde{\mathbb{R}P}^1$ we have $ (\bar{F}(x), \bar{F}(\delta x)) =  (\bar{F}(x), \delta\bar{F}(x))$,
hence $\bar{F} \circ \delta = \delta \circ \bar{F}$.

It follows that $\bar{F}$ induces a map $g$ on $\mathbb{R}P^1$.
Equation  \eqref{eq:F} and the equation
$F_\ast(\gamma) = h^{\alpha(\gamma)} \gamma$ \ above imply that
the induced map $g$ in $\mathbb{R}P^1$ commutes with the action
of $\bar{\Gamma}$ on $\mathbb{R}P^1$.
%commuting with the action of $\bar{\Gamma}$.
Since fixed points of elements of $\bar{\Gamma}$ in $\mathbb{R}P^1$ are dense, 
such a map $g$ is necessarily the identity
map of $\mathbb{R}P^1$. It follows that $\bar{F}$ is equal to some power $\delta^k$. In other words, the map $F$ coincides with the action on $\oo$ of some power $h^k$.
Equation \eqref{eq:F} implies that $F_\ast$ is trivial,
that is $F_\ast$ is the identity map.
Hence the morphism $\alpha$ is trivial,
and $\Gamma'$ coincides with $\Gamma$.
% modulo its $\operatorname{Mod}^{\pm}(\Sigma)^0$ orbit.

The Proposition is proved.
\end{proof}

\section{Counting the number of orbits of the action of the modular group on the set of finite index subgroups}

It follows from Proposition \ref{pro:Fgroup} that the description of orbital equivalence classes of Anosov foliations in $M$ reduces to the description of
the action of $\operatorname{Mod}^\pm(\Sigma)^0 \approx$ $\operatorname{Mod}^\pm(\Sigma)$ on $\mathfrak{G}_n$. For that purpose, in this section,
we provide a convenient parametrization of $\mathfrak{G}_n$ by $\operatorname{H}^1(\bar{\Gamma}, \mathbb{Z}/n\mathbb{Z})$, so that 
the action of the modular group appears as preserving the natural affine structure of $\operatorname{H}^1(\bar{\Gamma}, \mathbb{Z}/n\mathbb{Z})$.
%A convenient way to describe  $\mathfrak{G}_n$  appears in the proof of Proposition \ref{prop:ModCtransitive}.

Let $\Gamma$, $\Gamma'$ be two elements of  $\mathfrak{G}_n$. 
The restrictions of $p^0_\ast: \widetilde{\Gamma} \to \bar{\Gamma}$ to $\Gamma$ and $\Gamma'$ are both surjective, with kernel
the subgroup $n\mathbb{Z}$ generated by $h^n$
of the center $\mathbb{Z}$ of $\widetilde{\Gamma}$.
For every element $\gamma$ of $\Gamma$, there is some integer
$\alpha(\gamma)$ such that $h^{\alpha({\gamma})}\gamma$ lies in $\Gamma'$. More precisely, two elements of $\Gamma'$ (or $\Gamma$)
project on the same element of $\bar{\Gamma}$
if and only if their ``difference'' is an iterate of $h^n$.

Therefore, the integer $\alpha({\gamma})$ is unique modulo $n$,
and only depends on the projection $p_\ast^0(\gamma)$. It defines an application
$$\bar{\alpha}: \bar{\Gamma} \to \mathbb{Z}/n\mathbb{Z}$$
and it is easy to check that this application is a morphism.
This element of $\operatorname{H}^1(\bar{\Gamma}, \mathbb{Z}/n\mathbb{Z})$ can be seen as the difference between the two elements $\Gamma$ and $\Gamma'$.

Now fix $\Gamma$ and let $\Gamma'$ vary in $\mathfrak{G}_n$. 
Each $\Gamma'$ in $\mathfrak{G}_n$ has an associated morphism $\bar{\alpha}$,
which for simplicity of notation we omit the dependence of $\bar{\alpha}$ on
$\Gamma'$.
We study the association $\Gamma' \to \bar{\alpha}$.
For any $\Gamma'$ in $\mathfrak{G}_n$, and 
for any $\gamma'_1, \gamma'_2 \in \Gamma'$ with $p^0_\ast(\gamma'_1)
= p^0_\ast(\gamma'_2)$ then $\gamma'_2 = h^{ni} \gamma'_1$ for
some $i$ integer,  and for any $i$ integer, such $\gamma'_2$ is in
$\Gamma'$. It follows that $\Gamma'$ is determined by $\bar{\alpha}$,
that is, the association $\Gamma'$ to its $\bar{\alpha}$ function is
injective.
On the other hand given a morphism $\bar{\alpha}: \bar{\Gamma} \to \mathbb{Z}/n \mathbb{Z}$,  one builds a subset $\Gamma'$
of $\tilde{\Gamma}$ as follows: for each $\bar{\gamma}$ in $\bar{\Gamma}$
choose $\gamma$ in $\Gamma$ with $p^0_\ast(\gamma) = \bar{\gamma}$
and consider all $\gamma' \in \tilde{\Gamma}$ so that 
$\gamma' = h^{a + ni} \gamma$, where $i$ is any integer and $a$
is a fixed integer so that $a
= \bar{\alpha}(\bar{\gamma}) \ mod \ n$. It is easy to see that
$\Gamma'$ is a subgroup of $\tilde{\Gamma}$ (using
that $\bar{\alpha}$ is a morphism), that it has
index $n$, and that $p^0_\ast(\Gamma') = \bar{\Gamma}$.
It follows that the association $\Gamma'$ to $\bar{\alpha}$ is
surjective, so 
in fact this association is a bijection.

As a corollary, it is easy to see that
$\mathfrak{G}_n$ is an affine space with underlying
$\mathbb{Z}/n\mathbb{Z}$-module $\operatorname{H}^1(\bar{\Gamma}, \mathbb{Z}/n\mathbb{Z})$ (which is a genuine vector space
when $n$ is prime for example), and it follows from the discussion just before Proposition \ref{pro:Fgroup} that the action of Mod$^\pm(\Sigma)^0$ on $\mathfrak{G}_n$
preserves this affine structure.
It is an affine structure because there is no obvious origin $-$
recall that we picked an element $\Gamma \in \mathfrak{G}_n$ to
start with.

\begin{remark}\label{rk:ZnZ}
  {\em
  There is another way to express this identification between $\mathfrak{G}_n$ and $\operatorname{H}^1(\bar{\Gamma}, \mathbb{Z}/n\mathbb{Z})$.
  Let $\hat{\Gamma}_n$ be the quotient of $\widetilde{\Gamma}$ by the normal subgroup generated by $h^n$. There is an exact sequence:
\begin{equation}\label{eq:splitbar}
  0 \to \mathbb{Z}/n\mathbb{Z} \to \hat{\Gamma}_n \to \bar{\Gamma} \to 1
\end{equation}

\noindent
  which now is split since $n$ divides the Euler class. Then, the projection in
  $\hat{\Gamma}_n$ of an element $\Gamma$ of $\mathfrak{G}_n$ is a subgroup whose projection to $\bar{\Gamma}$ is an isomorphism. Hence, it provides a splitting  of \eqref{eq:splitbar},
  and this provides a one-to-one correspondance between $\mathfrak{G}_n$ and the set of splittings of \eqref{eq:splitbar}. This last one
  is notoriously $\operatorname{H}^1(\bar{\Gamma}, \mathbb{Z}/n\mathbb{Z})$.
  }
\end{remark}

%\begin{remark}\label{rk:tautau}
%{\em
%The group  $\operatorname{Mod}(M_1(\Sigma))$ acts affinely on $\mathfrak{G}_n$. Once prescribed an origin $\Gamma$, an element $[f]$
%of  $\operatorname{Mod}(M_1(\Sigma))$ has a linear part $L([f])$ and a translation part $\tau([f])$ and can be written (denoting by $v$ the elements
%of  $\mathfrak{G}_n \approx \operatorname{H}^1(\bar{\Gamma}, \mathbb{Z}/n\mathbb{Z})$):
%$$[f](v) = L([f])v + \tau([f])$$
%The translation part $\tau([f])$ represents the element of $\mathfrak{G}_n$ on which the origin $\Gamma$ is sent under $[f]$.
%Therefore, the elements $[f]$ of $\operatorname{Mod}(M_1(\Sigma))$ for which $\tau([f])$ vanishes are precisely the elements preserving
%$\Gamma$, that is, those in the image of $\alpha$.

%The linear part  $L([f])$ is trivial (that is $L([f])$ is
%the identity map) for $[f]$ in $K_1$.
%Hence this linear part is of the form $\bar{L}([\bar{f}])$ where
%$[\bar{f}]$ is the projection of $[f]$ in $\operatorname{Mod}(M_1(\Sigma))$
%to $\operatorname{Mod}^\pm(\Sigma)$, and $\bar{L}: \operatorname{Mod}^\pm(\Sigma) \to \operatorname{GL}(2g, \mathbb{Z}/n\mathbb{Z})$
%is a representation, where $g$ is the genus
%of the surface.
%This representation is clearly the one induced by the classical action of $\operatorname{Mod}^\pm(\Sigma)$ on $\operatorname{H}^1(\bar{\Gamma}, \mathbb{Z}/n\mathbb{Z})$,
%whose kernel is the Torelli subgroup. Its image is the entire symplectic group Sp$(2g, \mathbb{Z}/n\mathbb{Z})$.

%}
%\end{remark}

%See discussion prior to Lemma \ref{prop:ModCtransitive}, providinga bijection between Anosov flows and coverings.
We start with still another way to describe elements of $\mathfrak{G}_n$, but this time well-suited for describing the action
of $\operatorname{Mod}^\pm(\Sigma)$ on $\mathfrak{G}_n$.

Consider the section $\sigma_0: \bar{\Gamma} \to \widetilde{\Gamma}$ defined in Remark \ref{rk:actionQ}. We recall that it is the map associating to every
element of $\bar{\gamma}$ in $\bar{\Gamma}$ the
element $\widetilde \gamma$ in
 $(p^0_\ast)^{-1}(\bar{\gamma})$  preserving an orbit of the lifted geodesic flow. For this, it is good to
realize $\bar{\Gamma}$ as an uniform lattice in PSL$(2, \mathbb{R})$, and $\widetilde{\Gamma}$ as a lattice in $\widetilde{\operatorname{PSL}}(2, \mathbb{R})$,
acting on  $\widetilde{\mathbb{R}P}^1$. For every $\bar{\gamma}$ in $\bar{\Gamma}$, $\sigma_0(\bar{\gamma})$ is the unique element
in $\widetilde{\Gamma}$ above $\bar{\gamma}$ admitting a fixed point in $\widetilde{\mathbb{R}P}^1$.

The section $\sigma_0$ is not a morphism, recall Remark \ref{rk:actionQ} (the difference between $h$ and $\delta$ is simply that the first generates the
center of $\widetilde{\Gamma}$ whereas $\delta$ generates the center of $\widetilde{\operatorname{PSL}}(2, \mathbb{R})$; they coincide once $\widetilde{\Gamma}$
identified as a lattice in $\widetilde{\operatorname{PSL}}(2, \mathbb{R})$).
Therefore
\begin{equation}\label{eq:cocycle}
  \sigma_0(\bar{\gamma}_1\bar{\gamma}_2) = h^{c(\bar{\gamma}_1,\bar{\gamma}_2)}\sigma_0(\bar{\gamma}_1)\sigma_0(\bar{\gamma}_2)
\end{equation}
where $c$ is the Euler cocycle. 
%We will later give some examples that show that this cocycle takes values in $\{-1, 0, +1 \}$.

For every $\bar{\gamma}$ in $\bar{\Gamma}$ the element $\sigma_0(\bar{\gamma})^{-1}$ preserves the same orbit than $\sigma_0(\bar{\gamma})$, hence:
\begin{equation}\label{eq:inverse}
  \sigma_0(\bar{\gamma}^{-1}) = \sigma_0(\bar{\gamma})^{-1}
\end{equation}

Moreover, $\sigma_0$ is equivariant with respect to inner automorphisms: for every $\gamma_0$ in $\widetilde{\Gamma}$ we have:
\begin{equation}\label{eq:commute}
  \sigma_0(\bar{\gamma}_0\bar{\gamma}\bar{\gamma}_0^{-1}) = \gamma_0\sigma_0(\bar{\gamma})\gamma_0^{-1}
\end{equation}
where $\bar{\gamma}_0 = p_\ast^0(\gamma_0)$. Indeed, the conjugate $\gamma_0\sigma_0(\bar{\gamma})\gamma_0^{-1}$ preserves the image
under $\gamma_0$ of the orbit preserved by $\sigma_0(\bar{\gamma})$.

Let us now see how to parametrize $\mathfrak{G}_n$. Let $\sigma: \bar{\Gamma} \to \hat{\Gamma}_n$ be a morphism representing
an element of $\mathfrak{G}_n$ (cf. Remark \ref{rk:ZnZ}). Then, there is a map $\nu: \bar{\Gamma} \to \mathbb{Z}/n\mathbb{Z}$ defined by:
$$\sigma(\bar{\gamma}) =\bar{\sigma}_0(\bar{\gamma})\bar{h}^{\nu(\bar{\gamma})}$$
where $\bar{\sigma}_0(\bar{\gamma})$ and $\bar{h}$ are the projections in $\hat{\Gamma}_n$ of $\sigma_0(\bar{\gamma})$ and $h$.

As was the case for the map $\bar{\alpha}$ it follows that the
association $\Gamma' \in \mathfrak{G}_n$ to $\nu$ is injective.
The map $\nu$ is more complicated than the map $\bar{\alpha}$ in
the following ways:

First, the map $\nu$ is not a morphism, but must satisfy the following equation, for every pair $(\bar{\gamma}_1, \bar{\gamma}_2)$ of elements of
$\bar{\Gamma}$:
\begin{equation}\label{eq:nuc}
\nu(\bar{\gamma}_1\bar{\gamma}_2) = \nu(\bar{\gamma}_1) + \nu(\bar{\gamma}_2) - c(\bar{\gamma}_1, \bar{\gamma}_2),
\end{equation}
\noindent
where the equation should be understood in $\mathbb{Z}/n \mathbb{Z}$.
In other words, we are equating $c(\bar{\gamma_1},\bar{\gamma_2})$
which is an integer, with its projection in
$\mathbb{Z}/n \mathbb{Z}$.
The equation above means
that the coboundary of the $1$-cochain $\nu$ is the $2$-cocycle $c$: the Euler class of the surface with coefficients in
$\mathbb{Z}/n\mathbb{Z}$ is indeed trivial, i.e. a coboundary.

We see once again the affine space structure of $\mathfrak{G}_n$ over
$\operatorname{H}^1(\bar{\Gamma}, \mathbb{Z}/n\mathbb{Z})$ since the difference of such cochains are morphisms.

It follows from \eqref{eq:inverse}  that we have:

\begin{equation}\label{eq:inverse2}
  \nu(\bar{\gamma}^{-1}) = -\nu(\bar{\gamma})
\end{equation}

Moreover, for any $\gamma_0$ in $\widetilde{\Gamma}$, we have $\sigma_0(\bar{\gamma}_0) = \gamma_0{h}^k$ for some $k$, where as usual
$\bar{\gamma}_0 = p_\ast^0(\gamma_0)$. Therefore,
we have the equality ${\gamma}_0\sigma_0(\bar{\gamma})\gamma_0^{-1} = \sigma_0(\bar{\gamma}_0)\sigma_0(\bar{\gamma})(\sigma_0(\bar{\gamma}_0))^{-1}$.
So by equation (\ref{eq:commute}) we have

$$\sigma_0(\bar{\gamma}_0 \bar{\gamma} \bar{\gamma}^{-1}) \ = \
\gamma_0 \sigma_0(\bar{\gamma}) \gamma_0^{-1} \ = \
\sigma(\bar{\gamma}_0) \sigma_0(\bar{\gamma}) (\sigma(\bar{\gamma}_0))^{-1}$$

\noindent
and this implies that

\begin{equation}\label{eq:commute2}
  \nu(\bar{\gamma}_0\bar{\gamma}\bar{\gamma}_0^{-1}) = \nu(\bar{\gamma})
\end{equation}

The next step is now to describe the affine action of $\operatorname{Mod}^\pm(\Sigma) \approx$ Out$(\bar{\Gamma})$ on  $\mathfrak{G}_n$
in the $\nu$-coordinates.We will do it by describing the action of any automorphism of $\bar{\Gamma}$.

Let $\bar{f}$ be a diffeomorphism of $\Sigma$, and
we denote by $[\bar{f}]$ its isotopy class, i.e. the element of $\operatorname{Mod}^\pm(\Sigma) \approx$ Out$(\bar{\Gamma})$  it represents.
Let $[\bar{f}]_\ast$ denote one automorphism of $\bar{\Gamma}$ representing the action of $[\bar{f}]$ on $\bar{\Gamma}$.

\begin{lemma} \label{le:fbetaaction}
  The action of the element $[\bar{f}]$ of $\operatorname{Mod}^\pm(\Sigma)$ on $\mathfrak{G}_n$, in terms of the maps $\nu: \bar{\Gamma} \to \mathbb{Z}/n\mathbb{Z}$
above, is:
  $$[\bar{f}]\ast\nu = \pm \nu \circ [\bar{f}]_\ast^{-1}$$
The $\pm$ depends on whether the element $[\bar{f}]$ is orientation preserving or not, i.e. in
$\operatorname{Mod}(\Sigma)$ (the sign is $+$), or in $\operatorname{Mod}^-(\Sigma)$
(the sign is $-$).
\end{lemma}

Observe that, due to equation \eqref{eq:commute2} the term $\nu \circ [\bar{f}]_\ast^{-1}$ is well-defined despite of the fact that $[\bar{f}]_\ast$ is well-defined only up to inner automorphisms.

\begin{proof}
According to Remark \ref{rk:mod0} the diagonal action of the selected representative of $[\bar{f}]_\ast$ on the set of triples $X$ defines an homeomorphism
$f$ of $M_1(\Sigma)$ that maps the geodesic flow $\Phi_0$ onto itself. It lifts to a map $\tilde{f}: \mi \to \mi$ permuting the orbits of $\widetilde{\Phi}_0$.
There is an automorphism  $[f]_\ast: \widetilde{\Gamma} \to \widetilde{\Gamma}$
so that $\tilde{f}$ is $[f]_\ast$-equivariant. 
Given $\tau$ in $\tilde{\Gamma}$ then 
$[f]_\ast(\tau)$ is the only element of $\tilde{\Gamma}$
so that for any $x \in \mi$, then
$$\tilde{f}(\tau(x)) \ \ = \ \ [f]_\ast(\tau)(\tilde{f}(x)).$$
\noindent
Once again, $[f]_\ast$ is well-defined only up to inner automorphisms; we select one representative. If necessary, we change
our previous choice of $[\bar{f}]_\ast$ so that it coincides with the automorphism of $\bar{\Gamma}$ induced by $[f]_\ast$.

For any element $\bar{\gamma}$ of $\bar{\Gamma}$, by definition, $\sigma_0(\bar{\gamma})$ is the only element of $\widetilde{\Gamma}$ above
$\bar{\gamma}$ preserving some orbit $\tilde{\theta}$ of $\widetilde{\Phi}_0$. Then, $ [f]_\ast(\sigma_0(\bar{\gamma}))$ is an element of $\widetilde{\Gamma}$ above $[\bar{f}]_\ast(\bar{\gamma})$ preserving some orbit of $\widetilde{\Phi}_0$, namely $\tilde{f}(\tilde{\theta})$. This follows from 
the definition of $[f]_\ast$. Therefore:

\begin{equation}\label{eq:sigmaphi}
  \sigma_0([\bar{f}]_\ast(\bar{\gamma})) = [f]_\ast(\sigma_0(\bar{\gamma}))
\end{equation}

%\footnote{S: I am having a major problem with formula \ref{eq:sigmaphi}. $[\bar{f}]$ is a mapping class, hence it does not define an element of $Aut(\Gamma)$ but rather an element of $Out(\Gamma)$ - as remarked by ourselves after Corollary \ref{cor:powerT3}. Hence it defines an automorphism modulo a conjugation. Yes, $\nu$ is invariant by conjugation by a formula defined later $(15)$, but it does not make sense to write $[\bar{f}](\bar{\gamma})$. We could fix one lift $f_1$ of $[\bar{f}]$ to $\mi$ and this acts on $\bar{\Gamma}$, as well as on $\pi_1(M_1(\Sigma))$ as done after Corollary \ref{cor:powerT3}. But then we have to remark that all of this makes the action of $[f]$ well defined on $\mathfrak{G}_n$. I have not changed any of this,  and left it to you to do it, you know better how to do it correctly. Understanding this clearly affects the computations after formula $(15)$. In particular we need to understand well up to Lemma \ref{le:fbetaaction}.}

\begin{figure}
  \centering
   \includegraphics[scale=0.8]{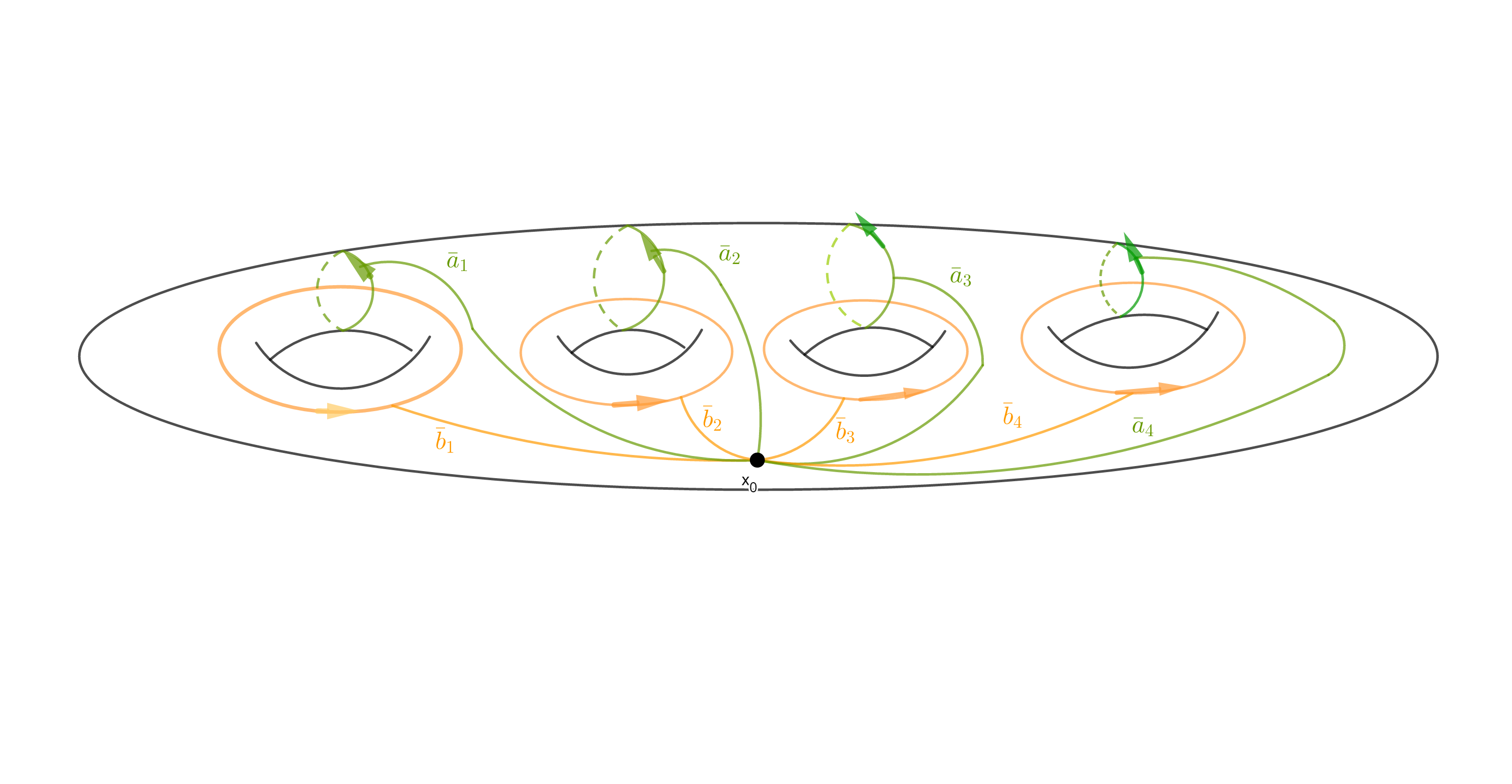}
\caption{A collection of oriented loops $(\bar{a}_1, \bar{b}_1, \ldots , \bar{a}_g, \bar{b}_g)$ forming a symplectic basis in homology, here for $g=4$.
The surface is considered has the boundary
of an handlebody in $\mathbb{R}^3$ and the orientation of the surface is given by the normal vector pointing out the handlebody.}
\label{fig:base}
\end{figure}

Let us now consider an element of $\mathfrak{G}_n$. We have seen two  ways to describe it:

-- either by some morphism $\sigma: \bar{\Gamma} \to \hat{\Gamma}_n$,

-- or a map $\nu: \bar{\Gamma} \to \mathbb{Z}/n\mathbb{Z}$.

The link between the two is given by the formula defining $\nu$:
$$\forall \bar{\gamma} \in \bar{\Gamma} \;\; \sigma(\bar{\gamma}) =\bar{\sigma}_0(\bar{\gamma})\bar{h}^{\nu(\bar{\gamma})}$$

The morphism map $[f]_\ast: \tilde{\Gamma} \to \tilde{\Gamma}$ preserves
the normal subgroup generated by $h^n$, therefore induces a morphism
$[f]_{\ast n}: \hat{\Gamma}_n \to \hat{\Gamma}_n$.

For any element $\Gamma$ of  $\mathfrak{G}_n$, the image under $[\bar{f}]$ of $\Gamma$  is the subgroup $[f]_\ast(\Gamma)$. It follows 
%It should be clear to the reader that
that if $\Gamma$ corresponds to the morphism $\sigma$, then $[f]_\ast(\Gamma)$ corresponds to the morphism $\sigma' = [f]_{\ast n} \circ \sigma \circ [\bar{f}]_\ast^{-1}$: 
First let $\ell: \tilde{\Gamma} \to \hat{\Gamma}_n$ be the projection.
We have that $\sigma(\bar{\Gamma}) = \ell(\Gamma)$. 
In addition
$[f]_\ast(\Gamma) = \Gamma'$. So it would be natural to consider the
map $[f]_{\ast n} \circ \sigma$ to be the morphism associated to $\Gamma'$.
However if we compose $[f]_{\ast n} \circ \sigma$ with the projection
from $\hat{\Gamma}_n$ to $\bar{\Gamma}$ we do not get the identity,
we get $[\bar{f}]_*$. Therefore precomposition
with $[\bar{f}]^{-1}_*$ has this property and is the map
associated with $\Gamma'$.

%Here we abuse notation and still denote by $[f]_\ast$ to the morphism
%induced by $[f]_\ast$ in $\hat{\Gamma}_n$.

On the other hand, let denote by $\nu'$ the map from $\bar{\Gamma}$ to $\mathbb{Z}/n\mathbb{Z}$ corresponding to  $[f]_\ast(\Gamma)$.
For any element $\bar{\gamma}$ of $\bar{\Gamma}$:

\begin{eqnarray*}
\bar{\sigma}_0(\bar{\gamma}) \bar{h}^{\nu'(\bar{\gamma})} & = & \sigma'(\bar{\gamma})\\
   &=& [f]_{\ast n}(\sigma([\bar{f}]_\ast^{-1}(\bar{\gamma}))) \\
   &=& [f]_{\ast n}(\bar{\sigma}_0([\bar{f}]_\ast^{-1}(\bar{\gamma}))\bar{h}^{\nu([\bar{f}]_\ast^{-1}(\bar{\gamma}))}) \\
   &=& [f]_{\ast n}(\bar{\sigma}_0([\bar{f}]_\ast^{-1}(\bar{\gamma}))) \bar{h}^{\pm \nu([\bar{f}]_\ast^{-1}({\gamma}))} \; \mbox{since} \; [f]_{\ast n} \;  \mbox{is a morphism preserving} \; <\bar{h}>\\ 
   &=& \bar{\sigma}_0(\bar{\gamma}) \bar{h}^{\pm \nu([\bar{f}]_\ast^{-1}(\bar{\gamma}))} \; \mbox{according to \eqref{eq:sigmaphi}}
\end{eqnarray*}
If $\bar{f}$ preserves orientation in $\Sigma$
then $f_\ast$ preserves $\bar{h}$, if $\bar{h}$ reverses orientation
then $f_\ast$ sends $\bar{h}$ to its inverse. 
Therefore, the sign in $\pm \nu([\bar{f}]_\ast^{-1}(\bar{\gamma}))$ is given by the fact that  $\bar{f}$ preserves orientation or not. The Lemma is proved. %\marginpar{T: I have changed this paragraph and corrected the equation just above}
%\marginpar{S: I am okay with it.}
\end{proof}

Let us now fix a generator system of $\bar{\Gamma}$ satisfying the usual presentation:
$$
\bar{\Gamma} = \langle \bar{a}_i, \bar{b}_i \; (i=1, ..., g) \: | \; [\bar{a}_1, \bar{b}_1][\bar{a}_2, \bar{b}_2] ... [\bar{a}_g, \bar{b}_g] = 1 \rangle
$$

More precisely, we select a base point $x_0$, and  take such a generator system so that every $\bar{a}_i$ and $\bar{b}_i$ is represented by a loop, such that
the homology classes $[\bar{a}_i]$ and $[\bar{b}_j]$ form a symplectic basis for the intersection form (i.e. the intersection numbers $[\bar{a}_i].[\bar{a}_j]$ and
$[\bar{b}_i].[\bar{b}_j]$ all vanish, and $[\bar{a}_i].[\bar{b}_j]$ as well except in the case $i=j$
where we have $[\bar{a}_i].[\bar{b}_i]=+1$). We furthermore require that the only intersection between any two such loops is the base point $x_0$, and another single intersection point between
the i-th loops $\bar{a}_i$ and $\bar{b}_i$.  See Figure \ref{fig:base}, for a collection of loops isotopic to one
satisfying these properties. 
%\marginpar{T: another modification}
%\marginpar{S: Also okay with  this.}
Denote by $a_i$ and $b_i$ the images of $\bar{a}_i$ and $\bar{b}_i$ by ${\sigma}_0$, respectively.

Lemma \ref{le:fbetaaction} describes the way that
$\operatorname{Mod}^\pm(\Sigma)$ acts on $\mathfrak{G}_n$, and we will now do the actual computation of this action.
For this, we use the explicit generating system for $\operatorname{Mod}(\Sigma)$  provided by the Lickorish Theorem.
Let $A_i$, $B_i$ and $C_j$ the simple closed curves depicted in Figure \ref{fig:likorish} ($1 \leq i \leq g$, $1 \leq j \leq g-1$).

\begin{figure}
  \centering
   \includegraphics[scale=0.8]{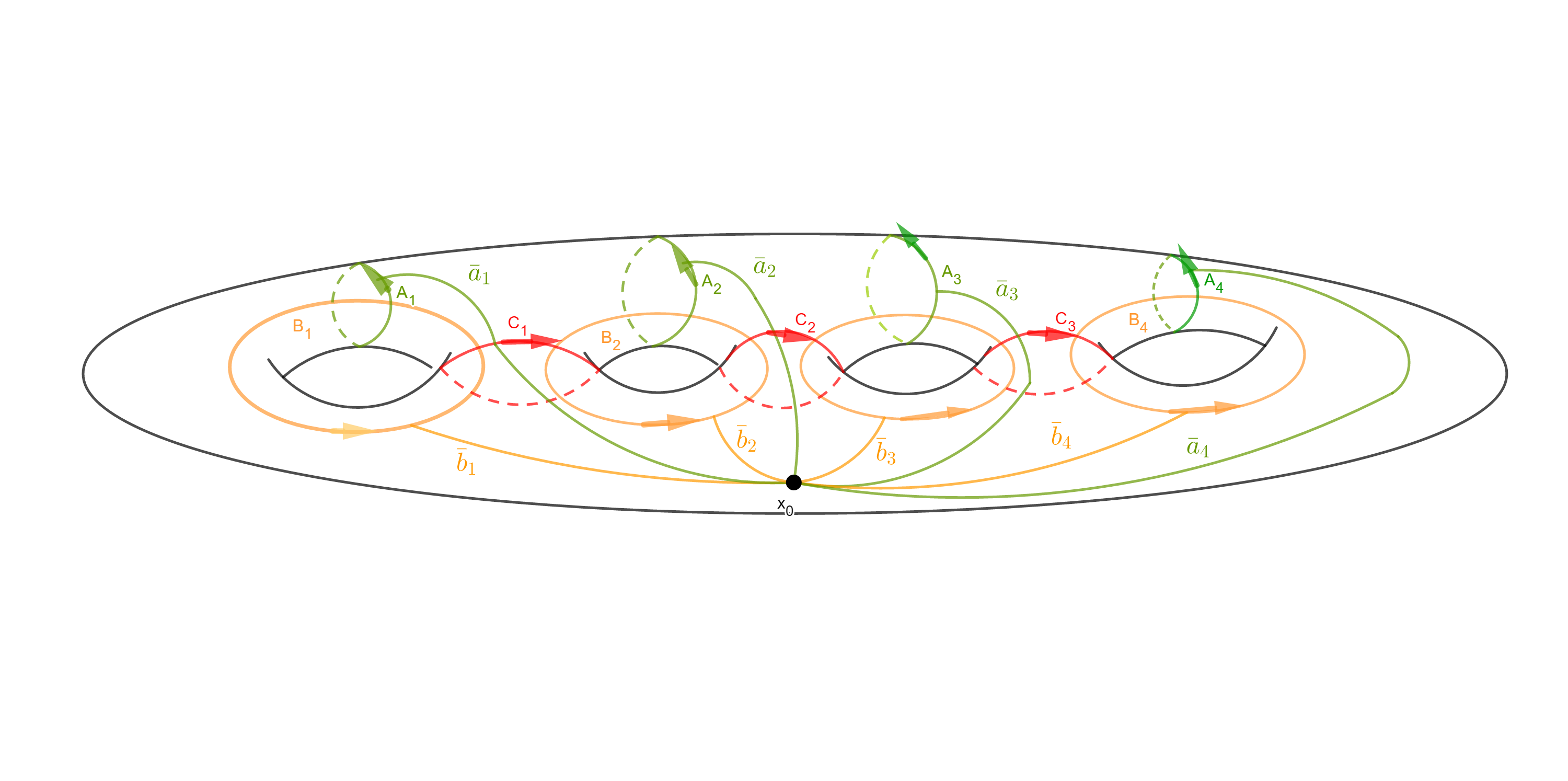}
\caption{According to Lickorish's Theorem, Dehn twists along these simple closed curves generate the mapping class group.}
\label{fig:likorish}
\end{figure}

\begin{theorem}[Lickorish \cite{licko}]\label{lick}
  Let $(A_1, B_1, ... , A_g, B_g, C_1, \ldots , C_{g-1})$ the system of simple closed curves depicted in
  figure \ref{fig:likorish}. Then, the Dehn twists along these curves generate $\operatorname{Mod}(\Sigma)$.
\end{theorem}

Slightly abusively, we will also denote by $A_i$, $B_i$ and $C_j$ the Dehn twists along these curves, and by $[A_i]_\ast$, $[B_i]_\ast$, $[C_j]_\ast$ the induced automorphisms
of $\bar{\Gamma} \approx \pi_1(\Sigma, x_0)$.

It follows from Figure \ref{fig:likorish} that for every $i$ between $1$ and $g$ that $[A_i]_\ast$ maps
every generator $\bar{a}_j$ and $\bar{b}_j$
onto itself, except $\bar{b}_i$ for which we have:
$$[A_i]_\ast(\bar{b}_i)=\bar{a}_i\bar{b}_i.$$
Therefore:
\begin{equation}\label{eq:calculA}
[A_i]^{-1}_\ast(\bar{b}_i)=\bar{a}^{-1}_i\bar{b}_i.
\end{equation}

Similarly $[B_i]_\ast$ maps every generator
onto itself except $\bar{a}_i$ for which we have:
$$[B_i]_\ast(\bar{a}_i) = \bar{a}_i\bar{b}^{-1}_i.$$
Hence:
\begin{equation}\label{eq:calculB}
[B_i]^{-1}_\ast(\bar{a}_i)=\bar{a}_i\bar{b}_i.
\end{equation}

\begin{figure}
  \centering
   \includegraphics[scale=0.8]{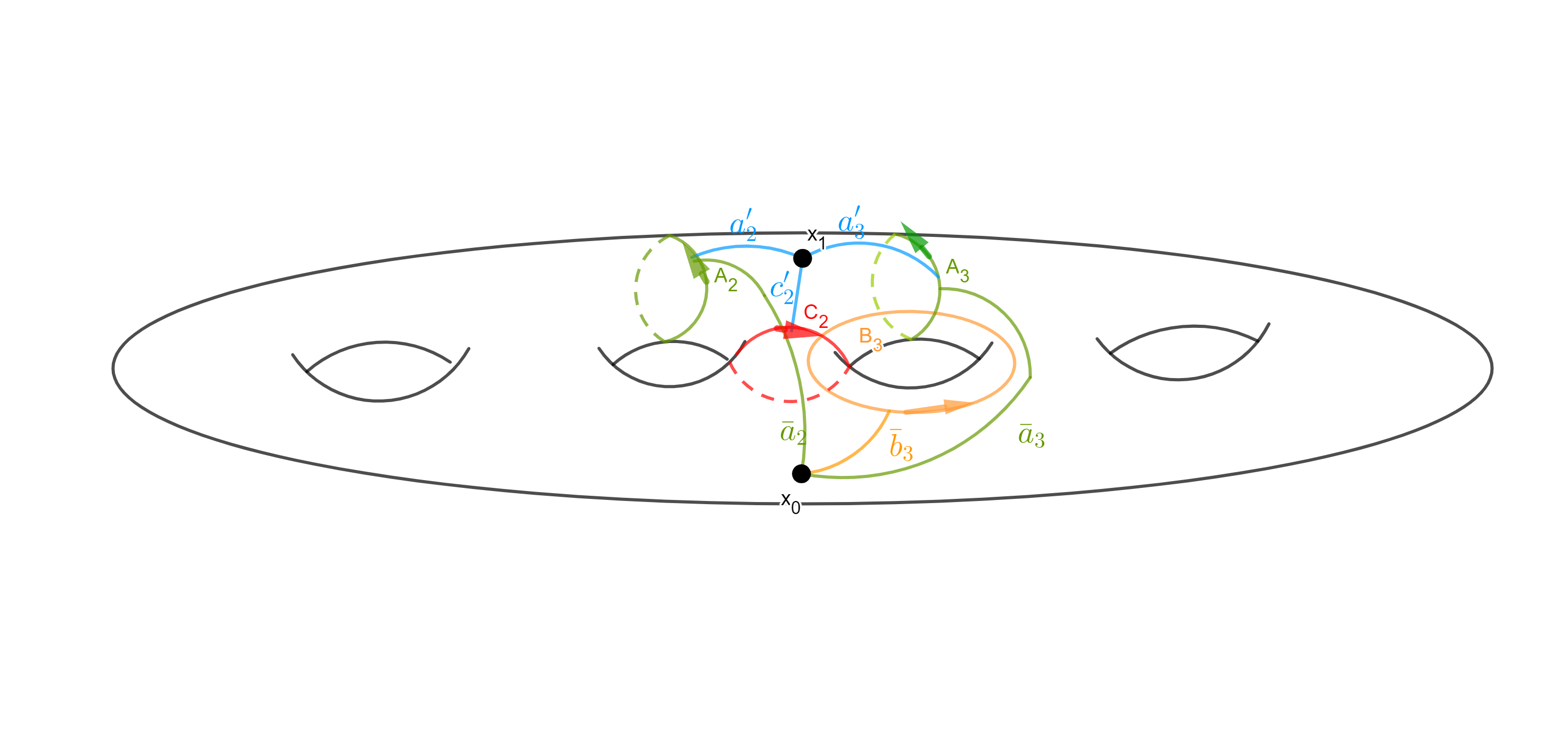}
\caption{Definition of $\bar{a}'_2$, $\bar{a}'_3$, $\bar{c}'_2$.}
\label{fig:a3}
\end{figure}

The computation of the action of $[C_i]_\ast$ ($1 \leq i \leq g-1$) is more intricate. Let $\bar{c}_i$ be the element of $\bar{\Gamma}$ defined
by the loop describing first the initial part of $\bar{a}_i$ until it reaches $C_i$, then following $C_i$ in the direction depicted in Figure \ref{fig:likorish},
then going back to $x_0$ by the path it uses to reach $C_i$. Then, $[C_i]_\ast$ acts trivially on every $\bar{a}_j$ for $j \neq i$, and also on every $\bar{b}_j$, except for $j=i$ and $j=i+1$. Furthermore:
$$[C_i]_\ast(\bar{a}_i) = \bar{c}_i\bar{a}_i\bar{c}^{-1}_{i}$$
$$[C_i]_\ast(\bar{b}_i) = \bar{c}_i\bar{b}_i$$
$$[C_i]_\ast(\bar{b}_{i+1}) = \bar{b}_{i+1}\bar{c}^{-1}_i$$

Therefore, the inverse of $[C_i]_\ast$ satisfies:
\begin{equation}\label{eq:calculC}
[C_i]^{-1}_\ast(\bar{a}_i) = \bar{c}^{-1}_i\bar{a}_i\bar{c}_i, \;\; [C_i]^{-1}_\ast(\bar{b}_i) = \bar{c}^{-1}_i\bar{b}_i, \;\; [C_i]^{-1}_\ast(\bar{b}_{i+1}) = \bar{b}_{i+1}\bar{c}_i.
\end{equation}

Now we claim the following equality:
\begin{equation}\label{eq:c3}
\bar{c}_i =\bar{a}_i\bar{b}_{i+1}^{-1}\bar{a}_{i+1}^{-1}\bar{b}_{i+1}
\end{equation}
%}\marginpar{T: I have changed the figure, the proof below, and the equality is not the same now... S: I agree with new equality.}

Indeed: introduce a new base point $x_1$ as depicted in  Figure \ref{fig:a3}, and paths $a'_i$, $a'_{i+1}$ and $c'_i.$ Let $\bar{a}'_{i}$ be the loop based at  $x_1$ 
following first $a'_i$, then turning around $A_i$, and going back to $x_1$ along $a'_i.$ Define similarly $\bar{a}'_{i+1}$ and $\bar{c}'_i.$ 
For simplicity we omit the dependence of $x_1$ on $i$.

It follows from this picture that we
have $\bar{c}'_i = \bar{a}'_i(\bar{a}'_{i+1})^{-1}$. 

Consider now the path $\ell$ going from $x_0$ to $x_1$ going along $\bar{a}_i$ until reaching $C_i$, and then reaching $x_1$ along $c'_i$. We isotope $c'_i$ so that it intersects $C_i$ at the same point that $\bar{a}_i$ does. Let $\bar{a}''_i$, $\bar{a}''_{i+1}$ and $\bar{c}''_i$ be the loops starting from $x_0$, following $\ell$, turning along $\bar{a}'_i$, $\bar{a}'_{i+1}$ and $\bar{c}'_i$ respectively before going back to $x_0$ along $\ell.$ Clearly, $\bar{a}''_i$ is homotopic to $\bar{a}_i$, and $\bar{c}''_i$ is homotopic to $\bar{c}_i$. Therefore, up to homotopy (based at $x_0$):
$$\bar{c}_i \approx \bar{a}_i(\bar{a}''_{i+1})^{-1}$$ 

But the same figure shows that $\bar{a}_{i+1}$ is homotopic to $ \bar{b}_{i+1}\bar{a}''_{i+1}\bar{b}_{i+1}^{-1}$.
 
The equality \eqref{eq:c3} follows.

We can now compute how the generators $A_i$, $B_i$, $C_i$ acts on $\mathfrak{G}_n$.

\begin{proposition}\label{pro:computeMod}
Let $i$ be an integer between $1$ and $g$.
The action of the element $[A_i]$ of $\operatorname{Mod}(\Sigma)$ represented by the Dehn twist $A_i$ on $\mathfrak{G}_n$ is given by:
$$[A_i]\ast\nu(\bar{a}_j) = \nu(\bar{a}_j) \ \  \forall j, \ \ \
[A_i]\ast\nu(\bar{b}_j)=\nu(\bar{b}_j)  \ \ \forall j \neq i, \\\ [A_i]\ast\nu(\bar{b}_i) = -\nu(\bar{a}_i) + \nu(\bar{b}_i).$$
The action of  the element $[B_i]$ of $\operatorname{Mod}(\Sigma)$ represented by $B_i$ is:
$$[B_i]\ast\nu(\bar{b}_j) = \nu(\bar{b}_j) \ \  \forall j, \ \ \
[B_i]\ast\nu(\bar{a}_j)=\nu(\bar{a}_j)  \ \ \forall j \neq i, \\\ [B_i]\ast\nu(\bar{a}_i) = \nu(\bar{a}_i) + \nu(\bar{b}_i).$$

Finally, for $i$ between $1$ and $g-1$, the action  of  the element $[C_i]$ of $\operatorname{Mod}(\Sigma)$ represented by $C_i$ is:
$$[C_i]\ast\nu(\bar{a}_j) = \nu(\bar{a}_j) \ \  \forall j, \ \ \
[C_i]\ast\nu(\bar{b}_j)=\nu(\bar{b}_j)  \ \ \forall j \neq i, i+1,$$
and
$$[C_i]\ast\nu(\bar{b}_i) = \nu(\bar{b}_i) -  \nu(\bar{a}_i) + \nu(\bar{a}_{i+1})+1, \;\; [C_i]\ast\nu(\bar{b}_{i+1}) = \nu(\bar{b}_{i+1}) +  \nu(\bar{a}_i) - \nu(\bar{a}_{i+1})-1$$

\end{proposition}

\begin{proof}
Recall that according to Lemma \ref{le:fbetaaction}, for every element $[\bar{f}]$ of $\operatorname{Mod}(\Sigma)$ we have:
$$[\bar{f}]\ast\nu = \nu \circ [\bar{f}]_\ast^{-1}.$$
Most formulae in the statement of the Proposition immediatly follow.
The only cases we have to consider are:
\begin{eqnarray*}
[A_i]\ast\nu(\bar{b}_i) &=& \nu(\bar{a}_i^{-1}\bar{b}_i)\;\; \mbox{(because of Equation \eqref{eq:calculA} )}\\
                                     &=& -\nu(\bar{a}_i) + \nu(\bar{b}_i) - c(\bar{a}_i^{-1}, \bar{b}_i) \;\; \mbox{ (because of Equation \eqref{eq:nuc} )}\\
                                    & = &-\nu(\bar{a}_i) + \nu(\bar{b}_i)\;\; \mbox{ (because of Lemma \ref{le:cvanishes} )}
\end{eqnarray*}

The formula for $[B_i]\ast\nu(\bar{a}_i)$ is similar.

\begin{figure}
  \centering
   \includegraphics[scale=0.6]{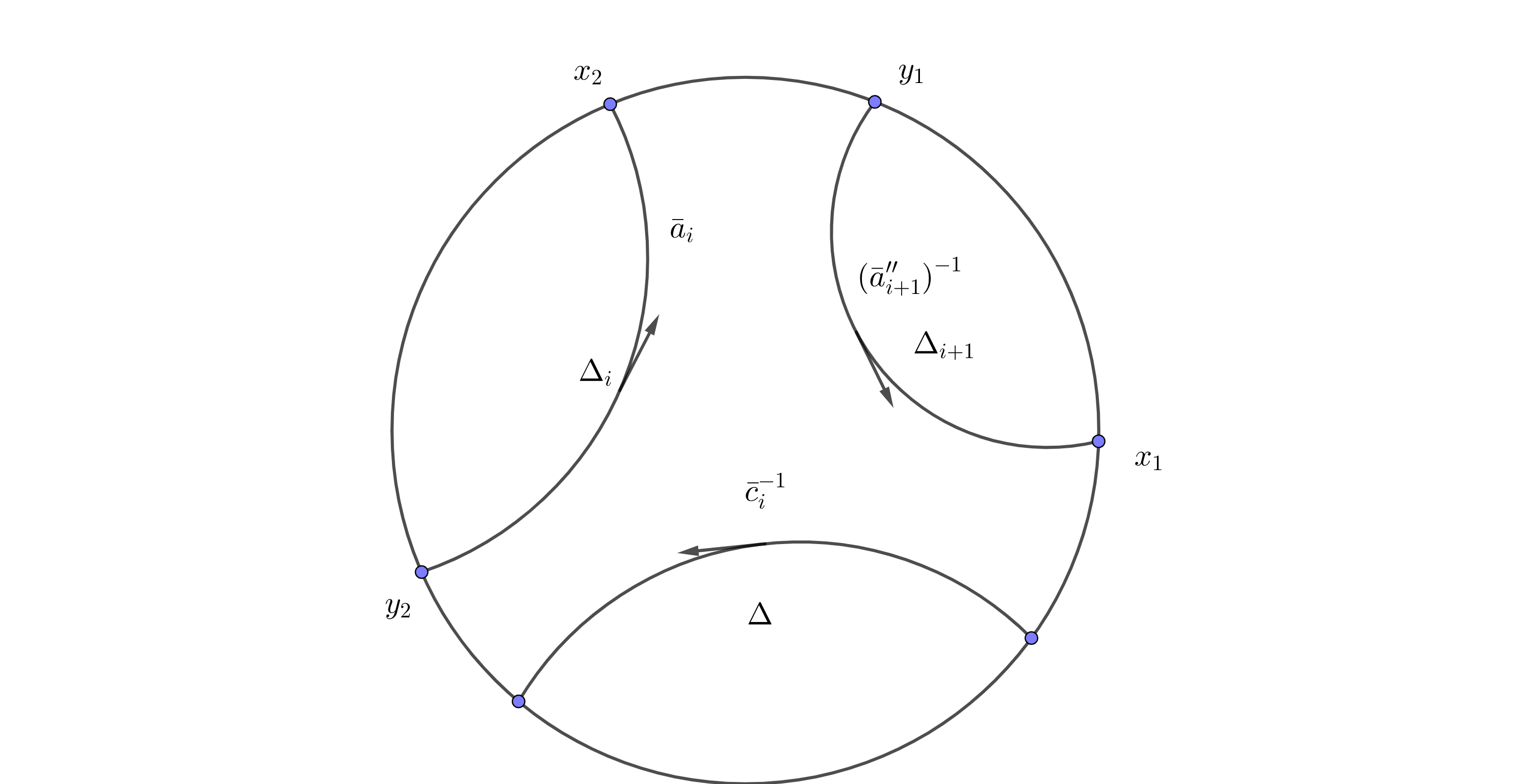}
\caption{$\Delta$, $\Delta_i$ and $\Delta_{i+1}$ are preserved by respectively $\bar{c}^{-1}$, $\bar{a}_i$ and $(\bar{a}''_{i+1})^{-1}$.
The attracting fixed point of $(\bar{a}''_{i+1})^{-1}$ is
$x_1$, and the attracting fixed point of $\bar{a}_i$ is $x_2$. }
\label{fig:pant}
\end{figure}

The remaining computations we have to do are for  $[C_i]_\ast\beta(\bar{a}_i)$, $[C_i]_\ast\beta(\bar{b}_i)$, and $[{C}_i]_\ast\beta(\bar{b}_{i+1})$.
Concerning the first, we have:
$$[C_i]\ast\nu(\bar{a}_i) = \nu(\bar{c}^{-1}_i\bar{a}_i\bar{c}_i) = \nu(\bar{a}_i)$$

Thanks to Equations \eqref{eq:calculC}, \eqref{eq:nuc} one gets:
$$[C_i]\ast\nu(\bar{b}_i) = \nu(\bar{b}_i) -\nu(\bar{c}_i) - c(\bar{c}_i^{-1}, \bar{b}_i), \\\ [C_i]\ast\nu(\bar{b}_{i+1}) = \nu(\bar{b}_{i+1}) + \nu(\bar{c}_i) - c(\bar{b}_{i+1}, \bar{c}_i).$$

But it follows from Lemma \ref{le:cvanishes} that $c(\bar{c}_i^{-1}, \bar{b}_i)$ and $ c(\bar{b}_{i+1}, \bar{c}_i)$ vanishes. Therefore:

\begin{equation}\label{eq:Cbb}
[C_i]\ast\nu(\bar{b}_i) = \nu(\bar{b}_i) -\nu(\bar{c}_i), \;\; [C_i]\ast\nu(\bar{b}_{i+1}) = \nu(\bar{b}_{i+1}) + \nu(\bar{c}_i)
\end{equation}

Hence, the key point is to compute $\nu(\bar{c}_i)$. Since $\bar{c}_i = \bar{a}_i(\bar{a}''_{i+1})^{-1}$, we have:
$$\nu(\bar{c}_i) = \nu(\bar{a}_i) - \nu(\bar{a}''_{i+1}) - c(\bar{a}_i, (\bar{a}''_{i+1})^{-1})$$

Hence, we have to compute $c(\bar{a}_i, (\bar{a}''_{i+1})^{-1})$.
Observe that the closed simple curves $A_i$, $C_i$ and $A_{i+1}$ form the boundary of a pair of pants. If we realize it as closed geodesics for
the hyperbolic metric, we see that they are the projections of the axis of respectively $\bar{a}_i$, $\bar{c}_i$ and $(\bar{a}''_{i+1})^{-1}$ (recall Figure \ref{fig:a3}). It follows that the axis of these elements in $\mathbb H^2$ have the configuration illustrated in Figure \ref{fig:pant}:

Consider now the lifted action in $\widetilde{\mathbb{R}P}^1$ is as in Figure \ref{fig:liftedpant}.
Here $a_i$ and $a''_{i+1}$ are the lifts of $\bar{a}_i$ and $\bar{a}''_{i+1}$ with fixed points
in $\widetilde{\mathbb{R}P}^1$.
Here $x_2, y_2$ are fixed by $\bar{a}_i$ with $x_2$ the
attracting fixed point. This generates $a_i$ fixing
$\widetilde x_2, \widetilde y_2$ (and infinitely many
other pairs).
Similarly $(\bar{a}''_{i+1})^{-1}$ fixes $x_1, y_1$ and so on.

\begin{figure}
  \centering
   \includegraphics[scale=0.8]{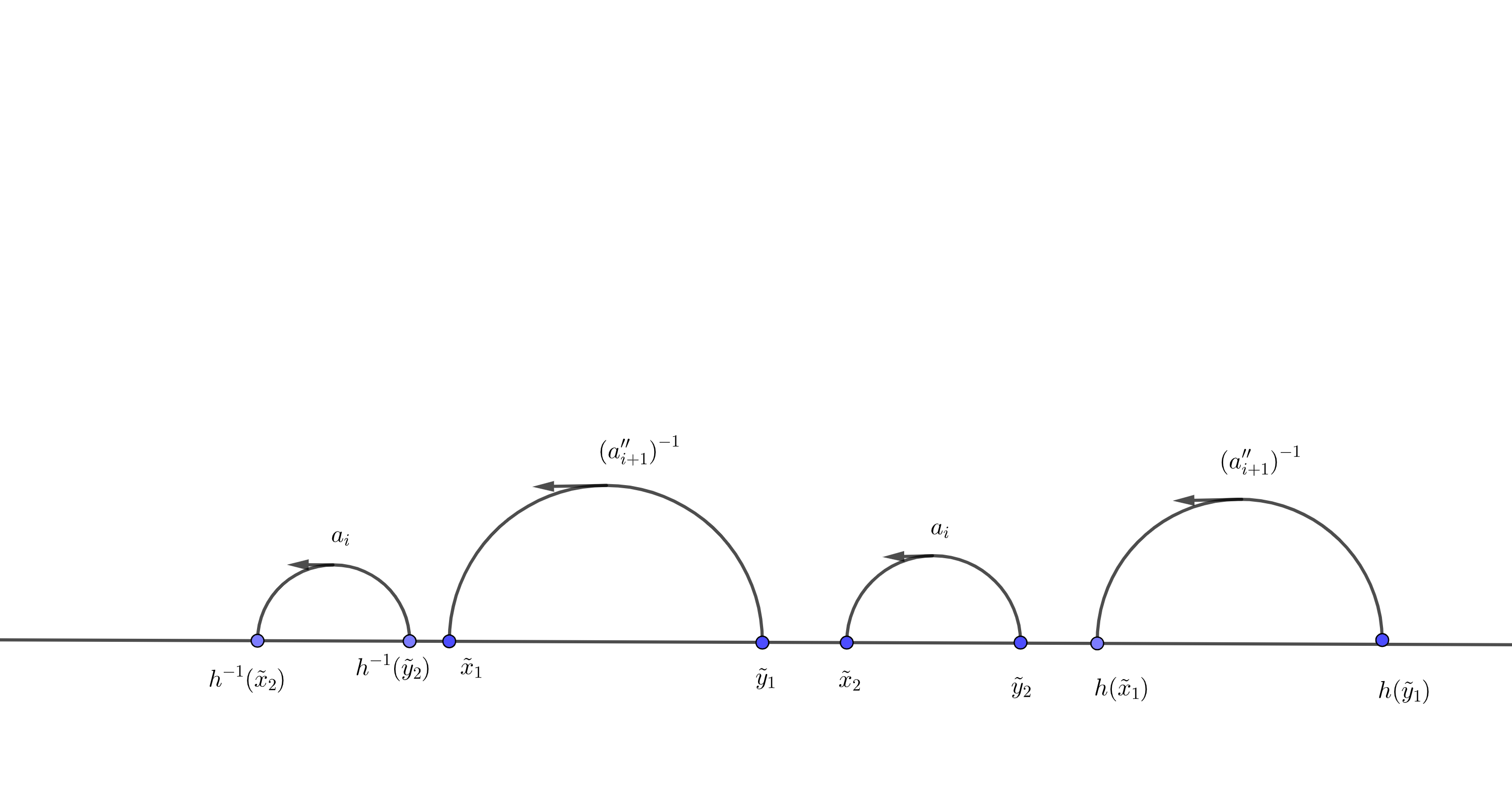}
\caption{Liftings in $\widetilde{\mathbb{R}P}^1$. We have depicted two attracting fixed points $\tilde{x}_1$ and $h(\tilde{x}_1)$ of
$({a}''_{i+1})^{-1}$, and two attracting fixed points $\tilde{x}_2$ and $h^{-1}(\tilde{x}_2)$ of  ${a}_i$.  }
\label{fig:liftedpant}
\end{figure}

\medskip
{\bf Claim:  for every $\tilde{x}$ in $\widetilde{\mathbb{R}P}^1$ we have $a_i(a''_{i+1})^{-1}(\tilde{x}) > \tilde{x}$.}
\medskip

Let us prove this claim. It is enough to prove it for every $\tilde{x}$ in the interval $[\tilde{y}_1, h(\tilde{y}_1)]$.

\begin{itemize}
\item \textbf{If $\tilde{x}$ lies in $[\tilde{y}_1, \tilde{x}_2]$:} Then, we have $({a}''_{i+1})^{-1}(\tilde{x}) > \tilde{x}$. If
$({a}''_{i+1})^{-1}(\tilde{x})$ lies in $[\tilde{y}_1, \tilde{x}_2]$, then its image by $a_i$ increases it, so the
final image is bigger
than $\tilde{x}$. If not, then  $({a}''_{i+1})^{-1}(\tilde{x}) \geq \tilde{x}_2$, and therefore $a_i({a}''_{i+1})^{-1}(\tilde{x})$ remains bigger or equal than $\tilde{x}$.
The Claim is true in this case.

\item \textbf{If $\tilde{x}$ lies in  $[\tilde{x}_2, \tilde{y}_2]$:}  The interval  $[\tilde{x}_2, \tilde{y}_2]$  is contained in the attracting region of $h(\tilde{x}_1)$ under $(a''_{i+1})^{-1}$
we therefore have $(a''_{i+1})^{-1}(\tilde{x}_2)>\tilde{x}_2$ and $(a''_{i+1})^{-1}(\tilde{y}_2) > \tilde{y}_2$. But since $\Delta_{i}$ is the lifting of a simple
closed geodesic, $\tilde{y}_2$ cannot be between $a_i(\tilde{x}_2)$ and $a_i(\tilde{y}_2)$. Therefore, we have $(a''_{i+1})^{-1}(\tilde{x}_2) > \tilde{y}_2$: the
interval $(a''_{i+1})^{-1}([\tilde{x}_2, \tilde{y}_2])$ is contained in the region $[\tilde{y}_2, h(\tilde{x}_1))]$ where we have $a_{i}(\tilde{x}) > \tilde{x}$.
Therefore, the Claim is still true there.

\item \textbf{If $\tilde{x}$ lies in $[\tilde{y}_2, h(\tilde{x}_1)]$:} then $(a''_{i+1})^{-1}(\tilde{x})$ lies in $[\tilde{x}, h(\tilde{x}_1)]$, and $a_i$ is increasing there, so the Claim is true in this case too.

\item \textbf{If $\tilde{x}$ lies in  $[h(\tilde{x}_1), h(\tilde{y}_1)]$:} Then, $(a''_{i+1})^{-1}(\tilde{x})$ remains in this interval. Therefore, its image under $a_i$ is bigger than $\tilde{x}$. 
Notice that $a_i$ is increasing in the interval
$[h(\tilde{x}_1, h(\tilde{y}_1)]$.
As in the second case we cannot have crossing of geodesics, so in fact
for every $\tilde{x}$ in $[h(\tilde{x}_1, h(\tilde{y}_1]$ then
$a_i (a''_{i+1})^{-1}(\tilde{x}) > h(\tilde{y}_1)$.
The Claim follows once more.
\end{itemize}

 The Claim is proved.

\vskip .1in
According to the Claim, the cocycle $c(\bar{a}_i, (\bar{a}''_{i+1})^{-1})$ is positive.
%, and since $c$ takes value in $\{-1, 0, +1\}$
In addition $a_i(\tilde{x}_1) < h(\tilde{x}_1)$, hence
$c(\bar{a}_i, (\bar{a}''_{i+1})^{-1}) \leq 1$. It follows
$c(\bar{a}_i, (\bar{a}''_{i+1})^{-1})=1$.

%In fact the same analysis show that given $\bar{\gamma}_1,\bar{\gamma}_2$
%with disjoint axes and $\gamma_i = \sigma_0(\bar{\gamma}_i)$, then the following happens:
%\begin{itemize}
%\item If the attracting fixed points of $\bar{\gamma}_1, \bar{\gamma}_2$ do not link the repelling fixed points we have $c(\bar{\gamma}_1,\bar{\gamma}_2) = 0$. This follow because if $I$ is the closure of an interval between a fixed of $\gamma_1$ and a fixed point of $\gamma_2$ not containing any other fixed points, then either $\gamma_1 \gamma_2(I)$ is strictly contained inside $I$ or strictly contains it. In either case $\gamma_1 \gamma_2$ has a fixed point, implying $c(\bar{\gamma}_1,\bar{\gamma}_2) = 0$.
% \item If the attracting fixed points of $\bar{\gamma}_1, \bar{\gamma}_2$ do not link the repelling fixed points we have that  $c(\bar{\gamma}_1,\bar{\gamma}_2)$ is either $+1$ or $-1$. If $x_1, y_1, x_2, y_2$ are as in the claim above and are cyclically ordered in $\mathbb{R}P^1$ then $c(\bar{\gamma}_1,\bar{\gamma}_2) = 1$. When $y_1, x_1, y_2, x_2$ are cyclically ordered,  $\mathbb{R}P^1$ then $c(\bar{\gamma}_1,\bar{\gamma}_2) = -1$. Notice that the $x'$s are the attracting fixed points, the $y'$s are the repelling fixed points.

Therefore:
$$\nu(\bar{c}_i) = \nu(\bar{a}_i) - \nu(\bar{a}''_{i+1})-1$$
and since $\bar{a}''_{i+1}$ is conjugated to $\bar{a}_{i+1}$, we get:
$$\nu(\bar{c}_i) = \nu(\bar{a}_i) - \nu(\bar{a}_{i+1})-1$$

 Equations \eqref{eq:Cbb} become:

$$[C_i]\ast\nu(\bar{b}_i) = \nu(\bar{b}_i) -  \nu(\bar{a}_i) + \nu(\bar{a}_{i+1})+1, \;\; [C_i]\ast\nu(\bar{b}_{i+1}) = \nu(\bar{b}_{i+1}) +  \nu(\bar{a}_i) - \nu(\bar{a}_{i+1})-1$$

Proposition \ref{pro:computeMod} is proved.
\end{proof}

Equation \eqref{eq:nuc} implies that the map 
$\nu: \bar{\Gamma} \to \mathbb{Z}/n\mathbb{Z}$ describing an element of $\mathfrak{G}_n$  is characterized by the values it takes on the generating set $\bar{a}_i, \bar{b}_i \; (i=1, ..., g)$. In other words, one can parametrize $\mathfrak{G}_n$ by the $n$-tuples $(\alpha_1, \beta_1, \ldots , \alpha_g, \beta_g)$
in $(\mathbb{Z}/n\mathbb{Z})^{2g}$ where $\alpha_i$ is the value taken by $\nu$
at $\bar{a}_i$, and $\beta_i$ the value it takes at $\bar{b}_i$.
Observe that all morphisms are possible since one can add to any $\nu$ any morphism
from $\bar{\Gamma}$ into $\mathbb{Z}/n\mathbb{Z}$. In particular, all the $\alpha_i$ and $\beta_i$ can vanish, this corresponds to the base group $\Gamma$.
%It will be our preferred element of $\mathfrak{G}_n$ in our computations thereafter.

  Let us decompose $\mathfrak{G}_n$ as a sum $V_1 \oplus \ldots \oplus V_g$ where every $V_i$ is made of elements with coordinates
  $(\alpha_1, \beta_1, \ldots , \alpha_g, \beta_g)$ (where
$\alpha_i, \beta_i \in \mathbb{Z}/n\mathbb{Z}$) satisfying $\alpha_j=\beta_j=0$ for all $j \neq i$. According to Proposition \ref{pro:computeMod},
this decomposition is invariant
  by the subgroup $G$ of $\operatorname{Mod}(\Sigma)$ generated by the $[A_i]$'s and the $[B_i]$'s. More precisely, in every $V_i$, the actions of $[A_i]$ and
  $[B_i]$ are given by the matrices:

  $$\left(
      \begin{array}{cc}
        1 & 0 \\
        -1 & 1 \\
      \end{array}
    \right) \quad
    \left(
      \begin{array}{cc}
        1 & 1 \\
        0 & 1 \\
      \end{array}
    \right)$$

%The action of $[A_i], [B_i]$ on $V_j$, $j \not = i$ is
%the identity.
   In particular, the action of $G$  preserves $(0, ... , 0)$: it is linear. Moreover, these two matrices generate the entire SL$(2, \mathbb{Z}/n\mathbb{Z})$, therefore, any element of  SL$(2, \mathbb{Z}/n\mathbb{Z})^g$  is realized by an element of $G$.
%Notice that the action of $[A_i], [B_i]$ on $V_j$ with $j \not = i$
%is the identity.

On the other hand, Proposition \ref{pro:computeMod} shows that for every $i$ between $1$ and $g-1$, the action of $[C_i]$ is trivial on $V_j$ for $j \neq i, i+1$, and that
on $V_i \oplus V_{i+1}$ this action is given by:

$$\left(
\begin{array}{ccccc}
1 & 0 & 0 & 0 & | 0 \\
-1 & 1 & 1 & 0 &  | 1 \\
0 &  0 & 1 & 0 & | 0 \\
1 & 0  & -1 & 1 & | -1
\end{array}\right)$$

This action is not linear, and its translation part $\tau([C_i]_\ast)$ in
the coordinate system $(\alpha_1, \beta_1, \ldots , \alpha_g, \beta_g)$ is the vector satisfying $\beta_i=-\beta_{i+1}=1$ for all $i$,
with all other coordinates vanishing.

We can now prove our Main Theorem:

\begin{theorem}\label{th:maincircle}
  Let $q_0: M \to M_1(\Sigma)$ be a finite covering of degree $n$ along the fibers. Then, if $n$ is odd, there is only one Anosov flow on $M$ up to orbital equivalence.
  If $n$ is even, there are exactly two orbital equivalence classes of Anosov flows.
\end{theorem}

\begin{remark}
{\em We thank J. Bowden who had indicated to us that this Theorem is closely related to Theorem $2.9$ in \cite{randal}, but we point out that our result concerns the
case where $\Sigma$ is a closed surface, whereas Randal-Williams only considers the case of compact surfaces with non-empty boundary.}
\end{remark}

\begin{proof}
  %According to Proposition \ref{pro:orbequi}, we just have to prove that the action of $\operatorname{Mod}(\Sigma)^0$ on $\mathfrak{G}_n$ is transitive.
The elements of $\mathfrak{G}_n$ are parametrized by $2g$ tuples
$(\alpha_1, \beta_1, ..., \alpha_g, \beta_g)$, with entries in
$\mathbb{Z}/ n \mathbb{Z}$. We first show that  any $\operatorname{Mod}(\Sigma)$ orbit
contains an element of the form $(0,0,....,0, \beta)$ for some $\beta$ in
$\mathbb{Z}/ n \mathbb{Z}$. From this, we show its orbit also has
the element
$(0,0,....,0, \beta +2)$. This reduces the number of orbits to $2$.
Then we analyze separately the cases of $n$ odd and $n$ even,
with $n$ even being trickier.

Since any element of  SL$(2, \mathbb{Z}/n\mathbb{Z})^g$  is realized by an element of $G$, every element of $\mathfrak{G}_n$ admits in its $G$-orbit an element of the form:
    \begin{equation}\label{eq:alphazero}
      (0, \beta_1, 0, \beta_2, \ldots , 0, \beta_g)
    \end{equation}

    Let us now consider the subgroup $H$ generated by the $g-1$ elements $C_1, \ldots, C_{g-1}$. Since all the curves $\bar{c}_i$ are disjoint,
    it is an abelian group. The action of some element $C_1^{k_1}C_2^{k_2}\ldots{C}_{g-1}^{k_{g-1}}$ is:
    \begin{eqnarray*}
      \alpha_i &\to& \alpha_i \\
      \beta_i &\to& \beta_i +k_i\alpha_{i+1} - (k_i+k_{i-1})\alpha_i+k_{i-1}\alpha_{i-1}+k_i-k_{i-1}
    \end{eqnarray*}
    where we adopt the convention $k_0=k_g=0$. This action does not change the value of each $\alpha_i$, and moreover one can check that
the sum $\sum_{i=1}^{g}\beta_i$
    is constant along the orbits of $H$. We do not do
an explicit computation here, which is fairly simple. Anyway, we see that by applying a well-suited element of $H$,
%\marginpar{T: I have changed $\mathfrak{Cov}$ by $H.$ \ S: Okay}
we can transform
    the element of the form \eqref{eq:alphazero} to another one where all the $\alpha_i$'s and $\beta_i$'s vanish, except $\beta_g$, since every $\alpha_i$ remains null and every $\beta_i$ is changed by adding $k_i - k_{i-1}.$ In other words,
    every element of $\mathfrak{G}_n$ admits in its $\operatorname{Mod}(\Sigma)$-orbit an element of the form:
        \begin{equation}\label{eq:alphabetazero}
      (0, 0,  \ldots , 0, 0, 0, \beta)
    \end{equation}

    Let us go back, applying $C_{g-1}$, we get:

$$
(0, 0,  \ldots , 0, 1, 0, \beta-1)
$$

Now use an element of $G$ mapping this element to:

$$
(0, 0,  \ldots , 0, 1, \beta-1, 0)
$$
%(meaning the old vanishing $\alpha_g$ has the value of the old $\beta_g$, and the new $\beta_g$ is now zero).

Apply $C_{g-1}$: we now get:

$$
(0, 0,  \ldots , 0, \beta+1, \beta-1, -\beta)
$$

Since $\beta-1$ and $-\beta$ are relatively prime, applying an element of $G$, composition of $A_{g}$'s and $B_{g}$'s,  we obtain:

$$
(0, 0,  \ldots , 0, \beta+1, 0, 1)
$$

and by applying $C_{g-1}^{-1-\beta}$, we get:

$$
(0, 0,  \ldots , 0, 0, 0, \beta+2)
$$

In summary, any element of $\mathfrak{G}_n$ admits in its  $\operatorname{Mod}(\Sigma)$-orbit an element of
the form $(0, 0,  \ldots , 0, 0, 0, \beta)$, and admits also all the elements of the form $(0, 0,  \ldots , 0, 0, 0, \beta+2k)$
where $k$ is any integer.

It follows that when $n$ is odd, in which case $2$ is a generator of $\mathbb{Z}/n\mathbb{Z}$, that every $\operatorname{Mod}(\Sigma)$-orbit
in $\mathfrak{G}_n$ contains $(0, 0,  \ldots , 0, 0, 0, 0)$. In other words, there is only one $\operatorname{Mod}(\Sigma)$-orbit; hence only  one $\operatorname{Mod}^\pm(\Sigma)$-orbit. The theorem
follows in this case from Proposition \ref{pro:Fgroup}.

\vskip .08in
Let us now assume that $n$ is even. The same argument as above shows that
there are at most two $\operatorname{Mod}(\Sigma)$-orbits in $\mathfrak{G}_n$:
the orbit of $(0, 0,  \ldots , 0, 0, 0, 0)$ and the orbit of $(0, 0,  \ldots , 0, 0, 0, 1)$. The proof of the Theorem will be finished
if
we prove that these two orbits are disjoint.

Since $n$ is even, every $\alpha_i$ and $\beta_i$, which is an integer modulo $n$, defines an element $\bar{\alpha}_i$ and $\bar{\beta}_i$ of
$\mathbb{Z}/2\mathbb{Z}$. In other words, there is a well-defined morphism $\mathfrak{G}_n \to \mathfrak{G}_2$.
This morphism is $\operatorname{Mod}^\pm(\Sigma)$-equivariant $-$ this
follows from Proposition \ref{pro:computeMod} which shows that
the action in $\mathfrak{G}_2$ is obtained from taking the action on $\mathfrak{G}_n$
and projecting it to modulo 2.
Therefore, if the $\operatorname{Mod}^\pm(\Sigma)$-orbits of $(0, 0,  \ldots , 0, 0, 0, 0)$ and of $(0, 0,  \ldots , 0, 0, 0, 1)$ are different in
$\mathfrak{G}_2$, they are also different in $\mathfrak{G}_n$.

Hence, in order to achieve the proof of the Theorem, we just have to consider the case $n=2$.

For every element of $\mathfrak{G}_2 = V_1 \oplus \ldots \oplus V_g$, let us call \emph{vanishing number} the number of indices $i$ for which
the components in $V_i$ are zero. More precisely, the vanishing number is this integer modulo $2$. For example, the vanishing number
for $(0, 0,  \ldots , 0, 0, 0, 0)$ is the class modulo $2$ of $0$, whereas the vanishing number
for $(0, 0,  \ldots , 0, 0, 0, 1)$ is the class modulo $2$ of $1$: they are different.

\vskip .1in
\noindent
{\bf {Claim:} The vanishing number is constant along $\operatorname{Mod}(\Sigma)$-orbits.}
\vskip .1in

Let us prove the claim: clearly, the vanishing number does not change under the action of $A_i$ or $B_i$, since they
act in $V_i$ as elements of SL$(2, \mathbb{Z}/2\mathbb{Z})$.

Let us consider $C_i$ ($1 \leq i \leq g-1$). It acts trivially on every $V_j$ except maybe $V_i$ and $V_{i+1}$.
We refer to the explicit formula for the action of $C_i$ on
$(\alpha_i, \beta_i, \alpha_{i+1}, \beta_{i+1})$ given in the
beginning of the proof of this theorem.

-- If $\alpha_i=\alpha_{i+1}=1$, the components in $V_i$ and $V_{i+1}$ are nonzero, and the same is true after applying
$C_i$ since we still have after the action $\alpha_i=\alpha_{i+1}=1$: the vanishing number remains the same.

-- If $\alpha_i \neq \alpha_{i+1}$, then, since we are in $\mathbb{Z}/2\mathbb{Z}$, we have $\alpha_{i+1}-\alpha_i+1=0$,
and it follows that $C_i$ acts trivially on such an element: the vanishing number remains the same in this case too.

-- If $\alpha_i=\alpha_{i+1}=0$: then we have:
$$C_i(0, \beta_i, 0, \beta_{i+1}) = (0, \beta_i +1, 0, \beta_{i+1}+1)$$
Hence, if $\beta_i=\beta_{i+1}=0$, or if $\beta_i=\beta_{i+1}=1$, the vanishing number before and after $C_i$ differ by $2$, hence is the same
modulo $2$. If $\beta_i \neq \beta_{i+1}$, one component vanishes and not the other, and the same is true after applying $C_i$.

We have proved the Claim. Since the vanishing numbers of $(0, 0,  \ldots , 0, 0, 0, 0)$ and $(0, 0,  \ldots , 0, 0, 0, 1)$ are
different, they cannot be in the same  $\operatorname{Mod}(\Sigma)$-orbit.

\vskip .08in

The only remaining step is to show that the vanishing number is also preserved by orientation reversing elements of  $\operatorname{Mod}^\pm(\Sigma)$.
For this, we just have to prove that it is true for \textit{one} of them. Let us consider the horizontal plane $P$ such that the surface is symetric relatively to the symetry $s$ in $P$. We can isotop every $B_i$ such that they are all contained in $P$, hence preserved by $s$ (including their orientation), and such that every $A_i$ is orthogonal to $P$ and preserved by $s$, but with the reversed orientation. We refer to figure \ref{fig:likorish}: we can choose $s$ so that it
preserves the loops $A_i, B_i$ in figure \ref{fig:likorish}. One can assume, and we do, that the base point $x_0$ is in $P$, hence fixed by $s$, and that for every $i$, the portion $\hat{a}_i$ (respectively $\hat{b}_i$) of the loop $\bar{a}_i$ (respectively $\bar{b}_i$) joining $x_0$ to $A_i$ (respectively $B_i$) lies above $P$, and, in addition, in the case of $\bar{a}_i$, this portion reaches $A_i$ at its intersection with $P.$

Then, the composition of $\hat{a}_i$ (respectively $\hat{b}_i$) with its image under $s$ provides a loop $\hat{a}_i^s$ (respectively $\hat{b}_i^s$) based at $x_0.$ 
It follows from these choices that $\bar{b}_i$ (respectively $\bar{a}_i$) is conjugated by $\hat{b}_i^s$ (respectively $\hat{a}_i^s$) to itself (respctively to its inverse).

Therefore, according to Lemma \ref{le:fbetaaction}, and since $s$ is orientation reversing, the action of $[s]$
on $\mathfrak{G}_2$ (or $\mathfrak{G}_n$) is given by, for all $j$:

$$[s]\ast \nu(a_j) =  \nu(a_j) $$
$$[s]\ast \nu(\bar{b}_j)  =  -\nu(\bar{b}_j)$$

This obviously preserves the vanishing number. Thus concludes the proof
of Theorem \ref{th:maincircle}.
\end{proof}

We end this Section with a proposition showing that the image in Aff$(\mathfrak{G}_n)$ of $\operatorname{Mod}(\Sigma)$ is isomorphic to Sp$(2g, \mathbb{Z}/n\mathbb{Z})$.

\begin{proposition}\label{pro:torellitrivial}
  The Torelli group acts trivially on $\mathfrak{G}_n$.
\end{proposition}

\begin{proof}
  By definition, an element of the Torelli group is an element of $\operatorname{Mod}^{\pm}(\Sigma)$ which acts trivially in homology. In particular it is orientation
preserving.
It also follows that its linear part 
is the identity. But it could still act on
  $\mathfrak{G}_n$ as a translation. However, by a Theorem of Johnson (\cite{johnson}) combined with a Theorem of Powell in the case $g=2$ (\cite{powell}),
  the Torelli group is generated by:

  -- Dehn twists along separating non-homotopically trivial simple closed curves in the case $g=2$,

  -- compositions $T_1T_2^{-1}$ where $T_1$ and $T_2$ are left Dehn twists along simple closed curves that are boundary
  of a compact surface of genus $1$  embedded in $\Sigma$ (in the case $g\geq3$).

  \begin{figure}
  \centering
   \includegraphics[scale=0.8]{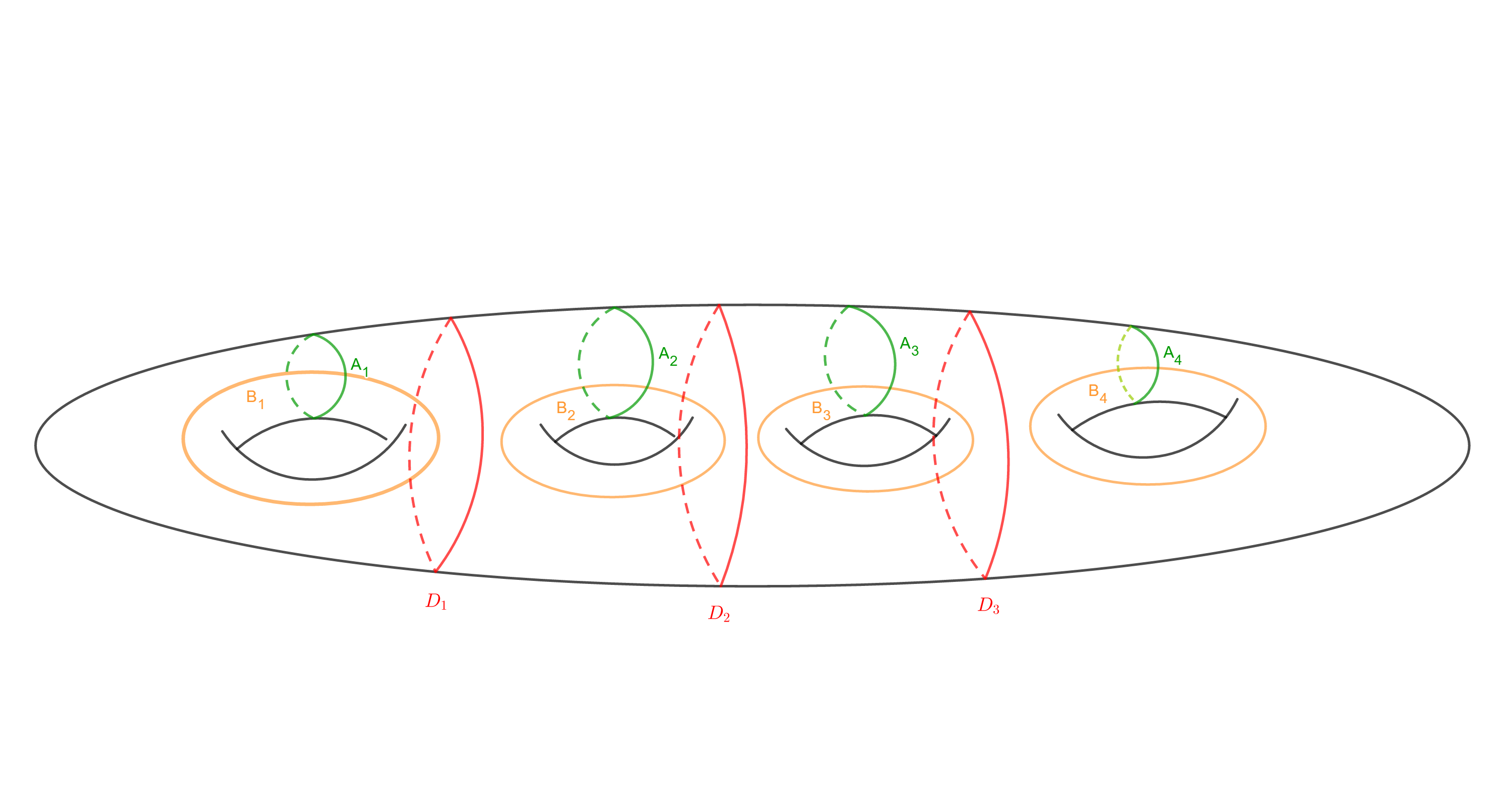}
\caption{$g-1$ non separating curves on the surface.}
\label{fig:baseD}
\end{figure}

  Consider the red simple closed curves $D_1, \ldots , D_{g-1}$ depicted in Figure \ref{fig:baseD}.  In the case $g=2$, there is only one such a curve $D_1$, and
every separating, homotopically non-trivial simple closed curve in $\Sigma$ is the image
  of $D_1$ under some homeomorphism. Therefore, it follows that in this case the Torelli group is generated by the conjugates of $T_1.$ 
%\marginpar{S: From this point on the Dehn twists were denoted by $D_i$.
%I changed them to $T_i$.}

In the case $g \geq 3$ any pair of simple closed curves bounding a genus $1$ surface embedded in $\Sigma$ is
  the image under some homeomorphism of the region of $\Sigma$
between two successive loops $D_i$ and $D_{i+1}$. Therefore, the Torelli group is generated by conjugates of one of the $T_i T_{i+1}^{-1}$.

Therefore, in order to prove the proposition we just have to show that every $T_i$ acts trivially on $\mathfrak{G}_n$.

But every $D_i$ are disjoint from the loops $A_i$
  and $B_i$. Therefore the Dehn twists along $D_i$ preserves the conjugacy classes of $\bar{a}_i$
  and $\bar{b}_i$. It follows that every $T_i$ acts trivially on $\mathfrak{G}_n$.

\end{proof}

\begin{remark}
  {\em
 % The algebraic content of Theorem \ref{th:maincircle} is that the action of $\operatorname{Mod}(\Sigma)^0$ on $\mathfrak{G}_n$ is transitive.
  %The $\operatorname{Mod}(M_1(\Sigma))$-stabilizer of every element of $\mathfrak{G}_n$ is isomorphic to $\operatorname{Mod}(\Sigma)$, but is never a conjugate of $\operatorname{Mod}(\Sigma)^0$.
  According to proposition \ref{pro:torellitrivial},  the action of $\operatorname{Mod}(\Sigma)$ on $\mathfrak{G}_n$ induces a (faithful) affine action of Sp$(2g, \mathbb{Z}/n\mathbb{Z})$ on the torus
  $\mathfrak{G}_n \approx (\mathbb{Z}/n\mathbb{Z})^{2g}$ whose linear part is the canonical representation of Sp$(2g,\mathbb{Z}/n\mathbb{Z})$.
  But this action is affine and admits a non-trivial translation part $[\tau]$, that represents a non-trivial element of
  $H^1(\operatorname{Sp}(2g,\mathbb{Z}/n\mathbb{Z}), (\mathbb{Z}/n\mathbb{Z})^{2g})$.

This feature is in contrast with the fact that $H^1(\operatorname{Sp}(2g,\mathbb{Z}), (\mathbb{Z})^{2g})$
  vanishes. 
%(every affine action of $\operatorname{Sp}(2g,\mathbb{R})$ has a global fixed point, since every linear representation splits as a sum 3  of irreducible representations).
}
\end{remark}

%Giroux was right at the end; with global fixed point in $\mathfrak{G}_n$ is moreover unique since the linear action of Sp$(4g, \mathbb{Z}$ on
%$\operatorname{H}^1(\bar{\Gamma}, \mathbb{Z}/n\mathbb{Z})$ has no nonzero fixed vector. It means that for this particular finite covering,
%every periodic orbit of the geodesic flow that projects in $\Sigma$ as a non disconnecting simple closed geodesic lifts as a periodic orbit in the finite covering,
%there is always an homeomorphism mapping such a non disconnecting simple closed curve to $a_1$.

%********************************************************************************

%WHAT FOLLOWS WOULD BE BETTER AS A DISCUSSION IN THE INTRODUCTION.

\begin{remark}\label{rk:giroux}
{\em 
In \cite{Giroux}, E. Giroux considered a related problem: the classification of contact
structures on the circle bundle $M$ up to isotopy and up to conjugation by a diffeomorphism. A case of interest for us is the case of universally tight
contact structures, with wrapping number $-n$. E. Giroux shows that they are all isotopic to the pull-back $q^\ast\alpha_0$
by some finite covering $q: M \to M_1(\Sigma)$ along the fibers,
where $\alpha_0$ is the contact $1$-form on $M_1(\Sigma)$ whose Reeb flow
is the geodesic flow (for some hyperbolic  Riemannian metric). According to Ghys' Theorem (see Theorem \ref{ghys}),
Anosov flows on $M$ are always isotopic to the Reeb flow of such a structure.
Clearly, an isotopy between contact structures provides an isotopy between
their Reeb flows, i.e. Anosov flows, and vice-versa. Similarly, there is a one-to-one correspondence between conjugacy classes of these contact structures
and orbital equivalence classes of Anosov foliations.

As a matter of fact, Giroux obtains a classification of isotopy classes similar as we do here: the action of $K_1$ on these contact structures is
simply transitive. Actually, it follows from Giroux's work and the results in the present paper that these two classification problems are equivalent.

Giroux furthermore claims to have classified conjugacy classes, and to show that the number of conjugacy classes of these contact structures is the number $\varphi(n)$
of divisors of $n$ (\cite[Theorem $3.1$ and Proposition $3.10$]{Giroux}).
Hence, this statement is in contradiction with our Theorem \ref{th:maincircle}.

The point is that Giroux's proof  is not completely correct at its concluding step: in the proof of Proposition $3.10$, he implicitly assumes that the action of $\operatorname{Mod}(\Sigma)$
on the space $\mathfrak{G}_n$ of index $n$ subgroups along
the fiber is linear, whereas this action is genuinely affine,
as shown in the present article.}
\end{remark}

{\footnotesize
{
\setlength{\baselineskip}{0.01cm}

}
}


\begin{thebibliography}{Be-Me}

{

%\setlength{\baselineskip}{0pt}
%\baselineskip0pt

\bibitem[An]{An} D. V. Anosov, {\em Geodesic flows on closed Riemannian
	manifolds with negative curvature}, Proc. Steklov Inst. Math.
	{\bf 90} (1969).

%\bibitem[An-Si]{An-Si} D. V. Anosov and Y. Sinai, {\em Some smoothly
%ergodic systems}, Russian Math. Surveys {\bf 22} (1967) 5
%	103-167.

%\bibitem[Ba1]{Ba1} T. Barbot, {\em G\'{e}om\'{e}trie transverse des flots},
%Thesis, \'{E}cole Normale Sup. Lyon, 1992.

\bibitem[Ba1]{Ba1} T. Barbot, {\em Caract\'{e}risation des flots d'Anosov en
	dimension $3$ par leurs feuilletages faibles}, Ergod. Th. Dynam. Sys.
	{\bf 15} (1995) 247-270.
	

\bibitem[Ba2]{Ba2} T. Barbot, {\em Flots d'Anosov sur les vari\'{e}t\'{e}s
	graph\'{e}es au sens de Waldhausen}, Ann. Inst. Fourier
	Grenoble {\bf 46} (1996) 1451-1517.
	
%\bibitem[Ba3]{Ba3} T. Barbot, {\em Mise en position optimale d'un tore
%par rapport \`{a} un flot d'Anosov}, Comm. Math. Helv.
%{\bf 70} (1995) 113-160.

%\bibitem[Ba4]{Ba4} T. Barbot, {\em Generalizations of Bonatti-Langevin
%example of Anosov flow and their classification up to topological
%	equivalence}, Comm. Anal. Geom. {\bf 6} (1998) 749-798.

%\bibitem[Ba5]{Ba5} T. Barbot, {\em Actions de groupes sur les
%	$1$-vari\'{e}t\'{e}s non s\'{e}par\'{e}es et feuilletages
%	de codimension un}, Ann. Fac. Sci. Toulose Math {\bf 7}
%	(1998) 559-597.

\bibitem[Ba3]{Ba3} T. Barbot, {\em Plane affine geometry of Anosov flows},
	Ann. Sci. Ecole Norm. Sup. {\bf 34} (2001) 871-889.

%\bibitem[Ba7]{Ba7} T. Barbot, {\em De l'hyperbolique au globalement
%hyperbolique}, M\'{e}moire d'habilitation, (2007), http://tel.archives-ouvertes.fr/tel-00011278/fr/

%\bibitem[Ba-Fe1]{bafe1} T. Barbot and S. Fenley, {\em Pseudo-Anosov flows in toroidal manifolds},
%Jour. Diff. Geom., {bf 17} (2013) 1877-1954.
%http://arxiv.org/abs/1007.0578
	
%\bibitem[Ba-Fe2]{BaFe2} T. Barbot and S. Fenley, {\em Classification and rigidity of totally periodic pseudo-
%Anosov flows in graph manifolds},
%Ergod. Th. Dynam. Sys., {\bf 35} (2015) 1681-1722.

%\bibitem[Ba-Fe3]{bafe3} T. Barbot and S. Fenley, {\em Free Seifert pieces of pseudo-Anosov flows},
%Geometry and Topology, {\bf 25} (2021) 1331-1440.
%submitted, https://arxiv.org/abs/1512.06341
	
\bibitem[Ba-Ma]{Ba-Ma} T. Barthelme, K. Mann, {\em Orbit equivalences
	of $\mathbb{R}$-covered Anosov flows and hyperbolic-like
	actions on the line}, to appear in Geom. Topol., arxiv: 2012.11811.

\bibitem[BFeM]{BFeM} T. Barthelme, S. Fenley, K. Mann, {\em Anosov flows
	with the same periodic orbits}, submitted, 
	arxiv: 2308.02098.

\bibitem[BFrM]{BFrM} T. Barthelme, S. Frankel, K. Mann, 
	{\em Orbit equivalence of pseudo-Anosov flows}, submitted,
	arxiv: 2211.10505.

%\bibitem[Be-Me]{Be-Me} M. Bestvina and G. Mess, {\em The boundary
%of negatively curved groups}, Jour. Amer. Math. Soc.,
%{\bf 4} (1991) 469-481.

%\bibitem[BBB]{BBB} F. B\'eguin, Y. Bin, C. Bonatti, {\em private communication.}

%\bibitem[Bo-La]{Bo-La} C. Bonatti and R. Langevin, {\em Un exemple de flot d'Anosov transitif transverse \`a un tore et non
%	conjugu\'e \`a une suspension}, Erg. Th. Dyn. Sys.
%	{\bf 14} (1994) 633-643.

%\bibitem[Bl-Ca]{Bl-Ca} S. Bleiler and A. Casson, {\em Automorphisms of
%	surfaces after Nielsen and Thurston}, Cambridge Univ. Press, 1988.

%\bibitem[Bow1]{Bow1} R. Bowen, {\em Periodic orbits for hyperbolic
%	flows}, Amer. Jour. Math. {\bf 94} (1970) 1-30.

%\bibitem[Bow2]{Bow2} R. Bowen, {\em Symbolic dynamics for hyperbolic
%	flows}, Amer. Jour. Math. {\bf 95} (1970) 429-459.

%\bibitem[Br]{Br} M. Brittenham, {Essential laminations in Seifert fibered spaces}, Topology {\bf 32} (1993) 61-85.


%\bibitem[Ca-Ju]{cassonjungreis} A. Casson and Doug Jungreis,  {\em Convergence groups and Seifert fibered manifold}, Inven. Math. {\bf 118} (1994) 441-456.


%\bibitem[BNR]{BNR} M. Brittenham, R. Naimi, R. Roberts, {\em Graph manifolds
%	and taut foliations}, Jour. Diff. Geom. {\bf 45} (1997) 446-470.

%\bibitem[Cal1]{Cal1} D. Calegari, {\em The geometry of $\rrrr$-covered
%	foliations}, Geometry and Topology {\bf 4} (2000) 457-515.

%\bibitem[Cal2]{Cal2} D. Calegari, {\em Foliations with one sided branching},
%	Geom. Ded. {\bf 96} (2003) 1-53.

%\bibitem[Cal3]{Cal3} D. Calegari, {\em Promoting essential laminations},
%	Inven. Math. {\bf 166} (2006) 583-643.

\bibitem[Ep]{Ep} Epstein, D. {\em Periodic flows on three-dimensional
	manifold}, Ann. of Math. {\bf 95} (1972) 66-82.

\bibitem[Fa-Ma]{farbmargalit} B. Farb, D. Margalit, {\em A primer on mapping class groups},
Princeton Mathematical Series {\bf 49} (2012).

\bibitem[Fen]{Fe1} S. Fenley, {\em Anosov flows in $3$-manifolds},
	Ann. of Math. {\bf 139} (1994) 79-115.

%\bibitem[Fe2]{Fe2} S. Fenley, {\em The structure of branching
%in Anosov flows of $3$-manifolds},
%	Comm. Math. Helv. {\bf 73} (1998) 259-297.

%\bibitem[Fe3]{Fe3} S. Fenley, {\em Foliations with good geometry},
%	Journal of the A.M.S. {\bf 12} (1999) 619-676.

%\bibitem[Fe4]{Fe4} S. Fenley, {\em Foliations and the topology of
%$3$-manifolds I: $\rrrr$-covered foliations and transverse pseudo-Anosov
%flows}, Comm. Math. Helv. {\bf 77} (2002) 415-490.

%\bibitem[Fe5]{Fe5} S. Fenley, {\em Pseudo-Anosov flows and inconpressible
%	tori}, Geom. Ded. {\bf 99} (2003) 61-102.
%
%\bibitem[Fe6]{Fe6} S. Fenley, {\em Laminar free hyperbolic $3$-manifolds},
%	Comm. Math. Helv. {\bf 82} (2007) 247-321.
%
%\bibitem[Fe7]{Fe7} S. Fenley, {\em Geometry of foliations and flows I:
%	Almost transverse pseudo-Anosov flows and asymptotic behavior
%	of foliations}, Jour. Diff. Geom. {\bf 81} (2009) 1-89.
%
%\bibitem[Fe8]{Fe8} S. Fenley, {\em Ideal boundaries of pseudo-Anosov flows
%	and uniform convergence groups with connections and applications
%	to large scale geometry}, arxiv:math/0507153.

\bibitem[Fe-Mo]{Fe-Mo} S. Fenley and L. Mosher, {\em Quasigeodesic
	flows in hyperbolic $3$-manifolds}, Topology
	{\bf 40} (2001) 503-537.
	
%\bibitem[Fo-Ha]{Fo-Ha} P. Foulon and B. Hasselblatt, {\em Contact Anosov flows on hyperbolic $3$-manifolds},
%Geometry and Topology {\bf 17} (2013) 1225-1252.

%\bibitem[Fra]{franks} J. Franks, {\em Anosov diffeomorphisms},
%Proc. Symp. Pure Math. {\bf 14} (1970), 61--93.

%\bibitem[Fr-Wi]{Fr-Wi} J. Franks and R. Williams, {\em Anomalous
%Anosov flows}, in Global theory of Dyn. Systems, Lecture Notes in Math. {\bf 819} Springer (1980).

%\bibitem[Fr]{Fr} D. Fried, {\em Transitive Anosov flows and pseudo-Anosov
%	maps}, Topology {\bf 22} (1983) 299-303.

%\bibitem[Ga]{gabaiconv} D. Gabai, {\em Convergence groups are Fuchsian groups},
%	Ann. Math. {\bf 136} (1992) 447-510.

%\bibitem[Ga1]{Ga1} D. Gabai, {\em Foliations and the topology of 3-manifolds},
%J. Diff. Geo. {\bf 18} (1983) 445--503.

%\bibitem[Ga2]{Ga2} D. Gabai, {\em Foliations and the topology
%of 3-manifolds II},
%J. Diff. Geo. {\bf 26} (1987) 461--478.

%\bibitem[Ga2]{Ga2} D. Gabai, {\em Foliations and the topology
%of 3-manifolds III},
%	J. Diff. Geo. {\bf 26} (1987) 479--536.

%\bibitem[Ga-Ka]{Ga-Ka} D. Gabai and W. Kazez, {\em Group negative curvature
%	for $3$-manifolds with genuine laminations}, Geom. Top.
%	{\bf 2} (1998) 65-77.

%\bibitem[Ga-Oe]{Ga-Oe} D. Gabai and U. Oertel, {\em Essential
%	laminations and $3$-manifolds}, Ann. of Math.
%	{\bf 130} (1989) 41-73.

%\bibitem[Gei]{Gei} Geiges, Hansj\"{o}rg, {\em An introduction to
%contact topology}, Cambridge Univ. Press, 2008.

\bibitem[Gh1]{Gh} E. Ghys, {\em Flots d'Anosov sur les
	$3$-vari\'{e}t\'{e}s fibr\'{e}es en cercles}, Ergod. Th
	and Dynam. Sys. {\bf 4} (1984) 67-80.

\bibitem[Gh2]{Gheuler} E. Ghys, {\em Groupes d'hom\'eomorphismes du cercle
et cohomologie born\'ee},
Contemporary Mathematics {\bf 58} (1987) 81-106.

%\bibitem[Gh3]{Ghdef} E. Ghys, {\em D\'eformations de flots d'Anosov et de groupes fuchsiens},
%Ann. Inst. Fourier. {\bf 42} (1992) 2019-247.

\bibitem[Gh3]{Ghacting} E. Ghys, {\em Groups acting on the circle},
Enseign. Math. (2) {\bf 47} (2001) 329--407.

\bibitem[Gi]{Giroux} E. Giroux, {\em Structures de contact sur les vari\'et\'es fibr\'ees en cercles au-dessus d'une surface},
Comm. Math. Helv. {\bf 76} (2001) 218--262.

%\bibitem[Go]{Go} S. Goodman, {\em Dehn surgery on Anosov flows},
%	Lec. Notes Math. {\bf 1007} (1983) 300-307, Springer.

\bibitem[Gro]{3remarks} M. Gromov, {\em Three remarks on geodesic dynamics and fundamental group},
Enseign. Math. (1)  {\bf 46} (2000) 391--402

%\bibitem[Hae]{Hae} A. Haefliger, {\em Group\"{o}ides d'holonomie
%	et classifiants}, Asterisque {\bf 116} (1984) 70-97.


%\bibitem[Han]{Han} M. Handel, {\em Global shadowing of pseudo-Anosov
%	homeomorphisms}, Ergod. Th. Dynam. Sys. {\bf 5} (1985) 373-377.

%\bibitem[Ha-Th]{Ha-Th} M. Handel and W. Thurston, {\em Anosov flows
%on new three manifolds}, Inv. Math. {\bf 59} (1980) 95-103.

\bibitem[Ha-Ka]{Ha-Ka} B. Hasselblatt and A. Katok, {\em Introduction
	to the modern theory of duynamical systems}, Cambridge Univ.
	Press, 1995.

\bibitem[Hu-Ka]{hurderkatok} S. Hurder, A. Katok, {\em Differentiability, rigidity and Godbillon-Vey classes for Anosov flows},
Inst. Hautes \'Etudes Sci. Publ. Math. {\bf 72} (1990), 5-61.

%\bibitem[He-Hi]{He-Hi} G. Hector and U. Hirsch, {\em Introduction to the geometry
%	of foliations}, Vieweg, 2nd Ed., 1986.

%\bibitem[He]{He} J. Hempel, {\em 3-manifolds}, Ann. of Math. Studies 86,
%	Princeton University Press, 1976.

%\bibitem[Ho]{holder} O. H\"older, {\em Die Axiome der Quantit\"at und die Lehre vom Mass.},
%Ber. Verh. Sachs. Ges. Wiss. Leipzig, Math. Phys. C1. {\bf 53} (1901) 1-64.


%\bibitem[Ja]{Ja} W. Jaco, {\em Lectures on three-manifold topology},
%	C.B.M.S. from A.M.S. {\bf 43} 1980.

%\bibitem[Ja-Sh]{Ja-Sh} W. Jaco and P. Shalen, {\em Seifert fibered spaces in $3$-manifolds},
%	Memoirs A. M. S. {\bf 220} 1979.

%\bibitem[Jo]{Jo} K. Johannson, {\em Homotopy equivalences of $3$-manifolds
%	with boundaries}, Lec. Notes Math, Springer {bf 761} (1979).

\bibitem[Joh]{johnson} D.L. Johnson, {\em  Homeomorphisms of a surface which act trivially on homology,}
Proc. Amer. Soc., {\bf 75} (1979) 119--125.

%\bibitem[Ke]{Ke} J. Kelley, {\em General topology}, Springer, 1955.

\bibitem[Lic]{licko} W. B. R. Lickorish, {\em  A finite set of generators for the homeotopy group of a $2$-manifold,}
Proc. Cambridge Philos. Soc., {\bf 60} (1964) 769--778.

%\bibitem[Mon]{monclairthesis} D. Monclair, {\em Dynamique Lorentzienne et Groupes de diff\'eomorphismes du cercle,}
%Ph. D thesis, ENS Lyon (2014).

%\bibitem[Mo1]{Mo1} L. Mosher, {\em Dynamical systems and the homology norm
%	of a $3$-manifold I. Efficient intersection of surfaces and flows},
%	Duke Math. Jour.  {\bf 65} (1992) 449--500.

%\bibitem[Mo2]{Mo2} L. Mosher, {\em Dynamical systems and the homology
%	norm of a $3$-manifold II},
%	Invent. Math. {\bf 107} (1992) 243--281.

%\bibitem[Mo3]{Mo3} L. Mosher, {\em Examples of quasigeodesic flows on
%hyperbolic $3$-manifolds}, in Proceedings of the Ohio State
%University Research Semester on Low-Dimensional topology,
%W. de Gruyter, 1992.

%\bibitem[Mo3]{Mo3} L. Mosher, {\em Laminations and flows transverse
%	to finite depth foliations}, manuscript available in the web
%	from	http://newark.rutgers.edu:80/~mosher/, Part I: Branched surfaces
%	and dynamics, Part II in preparation.

%\bibitem[No]{No} S. P. Novikov, {\em Topology of foliations}, Trans. Moscow
%	Math. Soc. {\bf 14} (1963) 268-305.

%\bibitem[Pa]{Pa} Palmeira, {\em Open manifolds foliated by planes},
%Ann. of Math., {\bf 107} (1978) 109-131.


%\bibitem[Pe1]{Pe1} G. Perelman, {\em The entropy formula for the Ricci flow
%and its geometric applications}, eprint, math.DG/0211159.

%\bibitem[Pe2]{Pe2} G. Perelman, {\em Ricci flow with surgery on three-manifolds},
%eprint, math.DG/0303109.

%\bibitem[Pe3]{Pe3} G. Perelman, {\em Finite extinction time for the solutions
%to the Ricci flow on certain three-manifolds}, eprint, math.DG/0307245.

%\bibitem[Pl2]{Pl2} J. Plante, {\em Foliations of $3$-manifolds with solvable
%fundamental group}, Inven. Math. {\bf 51} (1979) 219-230.


%\bibitem[Pl-Th]{Pl-Th} J. Plante and W. Thurston, {\em Anosov flows
%	and the fundamental group}, Topology {\bf 11} (1972) 147-150.

\bibitem[Pow]{powell} J. Powell, {\em  Two theorems on the mapping class group of a surface,}
Proc. Amer. Soc., {\bf 68} (1978) 347--350.

\bibitem[Rand]{randal} O. Randal-Williams, {\em Homology of the moduli spaces and mapping class groups of framed, $r$-Spin and Pin surfaces,}
J. Topol., {\bf 7} (2014), no. 1, 155--186.

%\bibitem[Ro-St]{Ro-St} R. Roberts and M. Stein, {\em Group actions on
%	order trees}, Topol. Appl. {\bf 115} (2001) 175-201.

%\bibitem[RSS]{RSS} R. Roberts, J. Shareshian and M. Stein, {\em Infinitely
%	many hyperbolic manifolds which contain no Reebless foliations},
%	Jour. Amer. Math. Soc. {\bf 16} (2003) 639-679.

%\bibitem[Sei]{Sei} H. Seifert, {\em Topologie dreidimensionaler
%gef\"aserter raume}, Acta Math. {\bf 60} 147-238.

%\bibitem[Ser]{serre} J.P. Serre, {\em Trees},
%Springer Monographs in Mathematics. Springer-Verlag, Berlin, 2003.

%\bibitem[Sh]{Sh} P. Shalen, {\em Three-manifolds and Baumslag-Solitar
%	groups}, ``Geometric topology and geometric group theory,
%	(Milwaukee, WI 1997), Top. Appl. {\bf 110} (2001) 113-118.

\bibitem[Sou]{Souto} J. Souto, {\em A remark on the action of the mapping class group on the unit tangent bundle},
Ann. Fac. Sci. Toulouse Math. {\bf 19} (2010)  589-601.

%\bibitem[Th1]{Th1} W. Thurston, {\em Foliations of $3$-manifolds
%that are circle bundles}, PhD thesis, University of
%California Berkeley, 1972.

%\bibitem[Th1]{Th1} W. Thurston, {\em The geometry and topology of 3-manifolds},
%	Princeton University Lecture Notes, 1982.	

%\bibitem[Th2]{Th2}  W. Thurston, {\em On the geometry and dynamics of
%	diffeomorphisms of surfaces}, Bull. A.M.S {\bf 19} (1988) 417-431.

%\bibitem[Th3]{Th3} W. Thurston, {\em Hyperbolic structures on $3$-manifolds II,
%	Surface groups and $3$-manifolds that fiber over the circle}, preprint.

%\bibitem[Th4]{Th4} W. Thurston, {\em Three manifolds, foliations and circles I},
%	preprint, 1997.

%\bibitem[Th5]{Th5} W. Thurston, {\em Three manifolds, foliations and circles II,
%	The transverse asymptotic geometry of foliations}, preprint, 1998.

%\bibitem[Th7]{Th7} W. Thurston, {\em Private communication.}

%\bibitem[Tuk]{tukia} P. Tukkia, {\em Homeomorphic conjugates of Fuchsian groups},
%J. Reine Angew. Math., {\bf 391} (1988), 1--54.

%\bibitem[Ve]{Ve} A. Verjovsky, {\em Codimension one Anosov flows}, Bol. Soc.
%Mat. Mex. {\bf 19} (1977) 49-77.

%\bibitem[Wa]{Wa} R. Waller, {\em Surfaces which are flow graphs}.

%\bibitem[Wald1]{wald1} F. Waldhausen, \emph{Eine Klasse von $3$-dimensionalen
%Mannigfaltigkeiten, I,\/} Invent. Math., {\bf 3} (1967), 308--333.

%\bibitem[Wald2]{wald2} F. Waldhausen, \emph{Eine Klasse von $3$-dimensionalen
%Mannigfaltigkeiten, II,\/} Invent. Math., {\bf 4} (1967), 87--117.

\bibitem[Wald]{waldlarge} F. Waldhausen, {\em On irreducible $3$-manifolds which are sufficiently large}, Ann. of Math. {\bf 87} (1968) 56-88.

}
\end{thebibliography}
\end{document}